\documentclass[12pt]{amsart}

\usepackage{tikz-cd} 
\usepackage{tensor}
\usepackage{accents}
\usepackage{appendix}
\usepackage{amsfonts}
\usepackage{amsmath}
\usepackage{amssymb}	
\usepackage{amsthm,bm}
\usepackage{array,booktabs,multirow}
\usepackage{braket}
\usepackage{centernot}
\usepackage{cite}
\usepackage{comment}
\usepackage{dsfont}
\usepackage{mathalfa}
\usepackage[shortlabels]{enumitem}
\usepackage{etoolbox}
\usepackage{float}
\usepackage[hang, flushmargin]{footmisc}
\usepackage{latexsym}
\usepackage{lipsum}
\usepackage{needspace}
\usepackage{tikz}
\usepackage{hyperref}
\usetikzlibrary{matrix,arrows}
\usepackage{diagbox}

\makeatletter
\def\namedlabel#1#2{\begingroup
   \def\@currentlabel{#2}%
   \label{#1}\endgroup
}
\makeatother

\theoremstyle{plain}
\newtheorem{thm}{Theorem}[section]
\newtheorem{cor}[thm]{Corollary}
\newtheorem{lem}[thm]{Lemma}
\newtheorem{prop}[thm]{Proposition}

\theoremstyle{definition}

\newtheorem{defn}[thm]{Definition}

\theoremstyle{remark}

\setlist[enumerate,1]{leftmargin=2em}

\def\A{\mathbf A}
\def\Aut{{\rm Aut}}
\def\BI{\mathfrak{BI}}
\def\C{\mathbb C}

\def\g{\mathfrak g}

\def\N{\mathbb N}

\def\T{\mathbf T}
\def\F{\mathbb F}
\def\Z{\mathbb Z}

\def\U{U(\mathfrak{sl}_2)}
\def\e{\varepsilon}

\newcommand{\floor}[1]{\left\lfloor #1 \right\rfloor}

\newcommand{\subsetdot}{\subset\mathrel{\mkern-7mu}\mathrel{\cdot}\,}

\title[A skew group ring, Leonard triples and odd graphs]{
A skew group ring of $\Z/2\Z$ over $U(\mathfrak{sl}_2)$, Leonard triples and odd graphs}

\author{Hau-Wen Huang}
\address{
Department of Mathematics\\
National Central University\\
Chung-Li 32001 Taiwan
}
\email{hauwenh@math.ncu.edu.tw}

\author{Chin-Yen Lee}
\address{
Department of Mathematics\\
National Central University\\
Chung-Li 32001 Taiwan
}
\email{109281001@cc.ncu.edu.tw}

\begin{document}

\begin{abstract} 
We employ a skew group ring of $\mathbb Z/2\mathbb Z$ over $U(\mathfrak{sl}_2)$ to construct modules over the universal Bannai--Ito algebra. In addition, we give the conditions under which the defining generators act as Leonard triples on the resulting modules. 
As a combinatorial realization, we establish an algebra homomorphism from the universal Bannai--Ito algebra onto the Terwilliger algebra of an odd graph. This homomorphism provides a unified description of Leonard triples on all irreducible modules over the Terwilliger algebra. 
\end{abstract}

\maketitle

{\footnotesize{\bf Keywords:}  Leonard triples, odd graphs, skew group rings, Terwilliger algebras}

{\footnotesize{\bf MSC2020:} 05E10, 05E30, 16S35, 33D45}

\allowdisplaybreaks

\section{Preliminaries on skew group rings and Hopf algebras}\label{s:skew}

Throughout this paper, we assume that all algebras are unital and associative, and that all algebra (anti-)homomorphisms are unital.  For any two elements $x,y$ of an algebra, we write $[x,y]=xy-yx$ and $\{x,y\}=xy+yx$.  Let $\N$ denote the set of nonnegative integers. 
By convention, a vacuous summation is equal to $0$ and a vacuous product is equal to $1$.

In analogy with the way semidirect products generalize the direct product of groups, the skew group rings may be regarded as a natural generalization of group rings. Let $R$ be a ring with unit $1$. Let $G$ be a finite group. Suppose that there is a group homomorphism $\varphi:G\to \Aut(R)$. The {\it skew group ring $R\rtimes_\varphi G$ of $G$ over $R$ induced by $\varphi$} consists of all formal sums
$$
\sum_{g\in G} a_g g
\qquad 
\hbox{where $a_g\in R$ and $g\in G$}.
$$
The addition operation is componentwise. The multiplication is distributively defined by 
\begin{gather}
\label{skew:product}
a_1 g_1\cdot a_2 g_2=\left(a_1\cdot\varphi(g_1)(a_2)\right) g_1g_2
\end{gather}
for all $a_1,a_2\in R$ and all $g_1,g_2\in G$. 
Each $a \in R$ can be identified with $a \cdot 1 \in R \rtimes_\varphi G$. Each $g \in G$ can be identified with $1 \cdot g \in R \rtimes_\varphi G$. 
Let $\F$ denote a field.
Suppose that $R$ is an algebra over $\F$. Then $R\rtimes_\varphi G$ is an algebra over $\F$ with scalar multiplication defined by
$$
c\left(\sum_{g\in G} a_g g\right)=\sum_{g\in G} (c a_g) g
$$
where $c\in \F$ and $a_g\in R$ for all $g\in G$.

Let $A$ denote an algebra over $\F$. The multiplication map $m:A\otimes A\to A$ is the $\F$-bilinear map defined by
$m(a\otimes b)=ab$ for all $a,b\in A$. 
The unit map $\iota : \F \to A$ is given by $\iota(c) = c \cdot 1$ for all $c \in \F$. 
Recall that $A$ is called a {\it Hopf algebra} if the following conditions hold:
\begin{enumerate}
\item[$\bullet$] There is an algebra homomorphism $\Delta:A\to A\otimes A$ such that 
$$
(1\otimes \Delta)\circ \Delta=(\Delta\otimes 1)\circ \Delta.
$$
The map $\Delta$ is called the {\it comultiplication} of $A$.

\item[$\bullet$] There is an algebra homomorphism $\e:A\to \F$ such that 
\begin{align*}
m\circ (1\otimes (\iota\circ\e))\circ \Delta=1,
\\
m\circ ((\iota\circ \e)\otimes 1)\circ \Delta=1.
\end{align*}
The map $\e$ is called the {\it counit} of $A$.

\item[$\bullet$] There is an algebra anti-homomorphism $S:A\to A$ such that 
\begin{align*}
m\circ (1\otimes S)\circ \Delta=\iota\circ\e,
\\
m\circ (S\otimes 1)\circ \Delta=\iota\circ\e.
\end{align*}
The map $S$ is called the {\it antipode} of $A$.
\end{enumerate}
In what follows, particular emphasis will be placed on two classical Hopf algebras and a Hopf algebra arising from their skew group ring.

Let $\g$ denote a finite-dimensional Lie algebra over the complex number field $\C$.
There exists a unique algebra $U(\g)$ over $\C$, up to isomorphism, together with a Lie algebra homomorphism $h:\g\to U(\g)$  
satisfying the following universal property: For any algebra $A$ over $\C$ and any Lie algebra homomorphism $\phi:\g\to A$, there exists a unique algebra homomorphism $\widetilde{\phi}:U(\g)\to A$ such that $\phi=\widetilde{\phi}\,\circ\, h$. The algebra $U(\g)$ is called the {\it universal enveloping algebra} of $\g$. 
By the Poincar\'{e}--Birkhoff--Witt theorem, the map $h$ is injective. Hence each element of $\g$ can be naturally identified with an element of $U(\g)$. The algebra $U(\g)$ is generated by all $u\in \g$. In addition, $U(\g)$ carries the following Hopf algebra structure:
\begin{enumerate}
\item[$\bullet$] The comultiplication $U(\g) \to U(\g) \otimes U(\g)$ is the algebra homomorphism that maps 
\begin{eqnarray*}
u &\mapsto &u\otimes 1+1\otimes u 
\qquad 
\hbox{for all $u\in \g$}.
\end{eqnarray*}

\item[$\bullet$] The counit $U(\g)\to \C$ is the algebra homomorphism that maps 
\begin{eqnarray*}
u &\mapsto & 0
\qquad 
\hbox{for all $u\in \g$}.
\end{eqnarray*}

\item[$\bullet$] The antipode $U(\g)\to U(\g)$ is the algebra anti-homomorphism that maps 
\begin{eqnarray*}
u &\mapsto &-u 
\qquad 
\hbox{for all $u\in \g$}.
\end{eqnarray*}
\end{enumerate}
Let $G$ be a group. The group ring $\C[G]$ of $G$ over $\C$ carries the following Hopf algebra structure:
\begin{enumerate}
\item[$\bullet$] The comultiplication $\C[G]\to \C[G]\otimes \C[G]$ is the algebra homomorphism
that maps 
\begin{eqnarray*}
g &\mapsto &g\otimes g
\qquad 
\hbox{for all $g\in G$}.
\end{eqnarray*}

\item[$\bullet$] The counit $\C[G]\to \C$ is the algebra homomorphism that maps 
\begin{eqnarray*}
g &\mapsto &1
\qquad 
\hbox{for all $g\in G$}.
\end{eqnarray*}

\item[$\bullet$] The antipode $\C[G]\to \C[G]$ is the algebra anti-automorphism 
that maps 
\begin{eqnarray*}
g &\mapsto &g^{-1}
\qquad 
\hbox{for all $g\in G$}.
\end{eqnarray*}
\end{enumerate}

We now assume that $G$ is a finite subgroup of ${\rm Aut}(\g)$. By the universal property of $U(\g)$,  
each $g\in G$ extends uniquely to an algebra automorphism of  $U(\g)$. Thus $G$ embeds naturally into ${\rm Aut}(U(\g))$. This gives rise to the skew group ring of $G$ over $U(\g)$, denoted by
$$
U(\g)_G:=U(\g)\rtimes G
$$ 
for brevity. 
In this case $U(\g)_G$ naturally inherits the Hopf algebra structure from those of $U(\g)$ and $\C[G]$. 
Specifically, the Hopf algebra structure of $U(\g)_G$ is as follows:
\begin{enumerate}
\item[$\bullet$] The comultiplication $\Delta:U(\g)_G\to U(\g)_G\otimes U(\g)_G$ is the algebra homomorphism defined by
\begin{align}
\Delta(u)&=u\otimes 1+1\otimes u
\quad \hbox{for all $u\in \g$},
\label{comul:1_UG}
\\
\Delta(g)&=g\otimes g
\qquad \hbox{for all $g\in G$}.
\label{comul:2_UG}
\end{align}

\item[$\bullet$] The counit $\e:U(\g)_G\to \C$ is the algebra homomorphism  defined by
\begin{align*}
\e(u)&=0
\qquad \hbox{for all $u\in \g$},
\\
\e(g)&=1
\qquad \hbox{for all $g\in G$}.
\end{align*}

\item[$\bullet$] The antipode $S:U(\g)_G\to U(\g)_G$ is the  algebra anti-automorphism defined by
\begin{align*}
S(u)&=-u
\qquad \hbox{for all $u\in \g$},
\\
S(g)&=g^{-1}
\qquad \hbox{for all $g\in G$}.
\end{align*}
\end{enumerate}
This section provides the necessary background on skew group rings and Hopf algebras.

\section{Introduction}\label{s:intro}

The Lie algebra $\mathfrak{sl}_2$ over $\C$ consists of all $2 \times 2$ complex trace-zero matrices. 
For a $Q$-polynomial distance-regular graph, the Terwilliger algebra with respect to a base vertex is the subalgebra of the full matrix algebra generated by the adjacency matrix and the corresponding dual adjacency matrix.

The interplay between $\mathfrak{sl}_2$ and Terwilliger algebras has been extensively investigated since the study of hypercubes \cite{hypercube2002}. More recently, the Clebsch--Gordan rule for $\U$ was realized through the decompositions of Terwilliger algebras of Hamming graphs \cite{Huang:CG&Hamming}. This line of work continued with a clarification of the subtle link between the Clebsch--Gordan coefficients of $\U$ and the Terwilliger algebras of Johnson graphs \cite{Huang:CG&Johnson}. A $q$-analogue of this link was then established for Grassmann graphs, relating the Clebsch--Gordan coefficients of $U_q(\mathfrak{sl}_2)$ to their Terwilliger algebras \cite{Huang:CG&Grassmann}. 
Related developments include the use of $U_q(\mathfrak{sl}_2)$ and $U_q(\widehat{\mathfrak{sl}}_2)$ to study the Terwilliger algebras of dual polar graphs and Grassmann graphs \cite{dualpolar:2022,Boyd:2013,Watanabe2020}.

In parallel, an algebra homomorphism from the universal Hahn algebra into $\U$ was established to capture the relation between the Terwilliger algebras of a hypercube and its halved graph \cite{halved:2023}. This line of work further led to an algebra homomorphism from the universal Racah algebra into $\U$, elucidating its connection with the Terwilliger algebras of halved cubes \cite{halved:2024}. 
Collectively, these results demonstrate a rich and fruitful set of links between $\mathfrak{sl}_2$ and Terwilliger algebras.

Building on these developments, this paper contributes to the study of the interplay between $\mathfrak{sl}_2$ and Terwilliger algebras by means of skew group rings. 
Recall that $\mathfrak{sl}_2$ admits the standard basis
\begin{gather}\label{EFH_matrices}
E=\begin{pmatrix}
0 &1
\\
0 &0
\end{pmatrix},
\quad 
F=\begin{pmatrix}
0 &0
\\
1 &0
\end{pmatrix},
\quad 
H=\begin{pmatrix}
1 &0
\\
0 &-1
\end{pmatrix}.
\end{gather}
There is an involutive Lie algebra automorphism $\rho:\mathfrak{sl}_2\to \mathfrak{sl}_2$ that sends 
\begin{eqnarray}
\label{rho}
E &\mapsto & F,
\qquad 
F\;\;\mapsto\;\; E,
\qquad 
H\;\;\mapsto\;\; -H.
\end{eqnarray}
The subgroup of ${\rm Aut}(\mathfrak{sl}_2)$ generated by $\rho$, which is isomorphic to $\Z/2\Z$, gives rise to the skew group ring 
$$
\U_{\Z/2\Z}.
$$
Let $d\geq 1$ denote an integer. 
Recall that the {\it odd graph} $O_{d+1}$ is the $Q$-polynomial distance-regular graph whose vertices are the $d$-element subsets of a $(2d+1)$-element set, with two vertices adjacent if and only if the corresponding subsets are disjoint. 
Extending the framework developed for hypercubes in  \cite{hypercube2002}, we establish a connection between the representations of 
$$
\U_{\Z/2\Z}^{\otimes 2}:=\U_{\Z/2\Z}\otimes \U_{\Z/2\Z}
$$ 
and the Terwilliger algebras of odd graphs, mediated by the universal Bannai--Ito algebra.

The Bannai--Ito polynomials were originally introduced in the monograph of Bannai and Ito on association schemes \cite{BannaiIto1984}. They were later characterized as the 
$q\to -1$ limit of the $q$-Racah polynomials \cite{Zhedanov:2012}. In parallel, a set of algebraic relations can be formulated, which are known as the Askey--Wilson relations associated with the Bannai--Ito polynomials \cite{Vidunas:2007,lp&awrelation}. The Bannai--Ito algebras are defined as algebras satisfying these relations. To provide a uniform framework for their structural study, the universal Bannai--Ito algebra was defined as follows: 
The {\it universal Bannai--Ito algebra} $\BI$ is an algebra over $\C$ generated by $X,Y,Z$ and the relations assert that each of 
\begin{align}
\label{BI:centers}
\{X,Y\}-Z,  
\qquad 
\{Y,Z\}-X, 
\qquad \{Z,X\}-Y  
\end{align}
is central in $\BI$ \cite{R&BI2015,Huang:R<BI}. In other words, the Bannai--Ito algebra is obtained by assigning scalar values to the defining central elements (\ref{BI:centers}) of $\BI$.

By (\ref{comul:1_UG}) and (\ref{comul:2_UG}) the comultiplication $\Delta:\U_{\Z/2\Z}\to \U_{\Z/2\Z}^{\otimes 2}$ of $\U_{\Z/2\Z}$ is an algebra homomorphism given by 
\begin{align}
\Delta(E)&=E\otimes 1+1\otimes E,
\label{DeltaE}
\\
\Delta(F)&=F\otimes 1+1\otimes F,
\label{DeltaF}
\\
\Delta(H)&=H\otimes 1+1\otimes H,
\label{DeltaH}
\\
\Delta(\rho)&=\rho\otimes \rho.
\label{Deltarho}
\end{align} 
For any $\U_{\Z/2\Z}^{\otimes 2}$-module $V$ and any scalar $\theta\in\C$, we define 
\begin{gather}
\label{V(theta)}
V(\theta)=\{v\in V\,|\,\Delta(H)v=\theta v\}.
\end{gather}
Among these, the case $\theta=1$ is of particular interest. 
In the first main result of this paper, we establish a $\BI$-module structure on $V(1)$.

\begin{thm}
\label{thm:V(1)}
Suppose that $V$ is a $\U_{\Z/2\Z}^{\otimes 2}$-module. 
Then $V(1)$ is a $\BI$-module given by
\begin{align}
    X&=\Delta(E\rho),
    \label{X:V(1)}\\
    Y&=\frac{1}{2}(H\otimes 1-1\otimes H),
    \label{Y:V(1)}\\
    Z&=(E\otimes 1-1\otimes E)\Delta(\rho).
    \label{Z:V(1)}
\end{align}
\end{thm}

Recall that a {\it Leonard pair} consists of two diagonalizable linear operators on a nonzero finite-dimensional vector space,  such that each acts in an irreducible tridiagonal fashion on an eigenbasis of the other.
As a natural extension,  a {\it Leonard triple} consists of three diagonalizable linear operators on a nonzero finite-dimensional vector space, such that any two act in an irreducible tridiagonal fashion on an eigenbasis of the third.
For the formal definitions, we refer the reader to \cite{lp2001, cur2007-2}.
Leonard pairs and Leonard triples form an important class of operators satisfying the Askey--Wilson relations \cite{LTRacah,Huang:2012,LT2013,LT2015,Vidunas:2007,lp&awrelation,halved:2024}. Consequently, they are naturally connected with $\BI$-modules. 
We clarify the position of Theorem \ref{thm:V(1)} within the framework of Leonard triples in the second main result:

\begin{thm}
\label{thm:V(1)&LT}
Suppose that $V$ is a finite-dimensional irreducible $\U_{\Z/2\Z}^{\otimes 2}$-module. 
Then the following conditions are equivalent:
\begin{enumerate}
\item $V(1)$ is nonzero.

\item The $\BI$-module $V(1)$ is irreducible.

\item $X,Y,Z$ act on the $\BI$-module $V(1)$ as a Leonard triple.
\end{enumerate}
\end{thm}

For any set $X$ let $\C^X$ denote the vector space over $\C$ with basis $X$.
Let $\Omega$ be a finite set and write $2^\Omega$ for the power set of $\Omega$. 
For any $x_0\in 2^\Omega$ we equip $\C^{2^\Omega}$ with a $\U_{\Z/2\Z}^{\otimes 2}$-module structure, denoted by 
$$
\C^{2^\Omega}(x_0).
$$ 
By Theorem \ref{thm:V(1)} this endows $\C^{2^\Omega}(x_0)(1)$ with a $\BI$-module structure.

We now assume that $\Omega$ is a $(2d+1)$-element set and view ${\Omega\choose d}$, the collection of $d$-element subsets of $\Omega$, as the vertex set of $O_{d+1}$. Let $\A$ denote the adjacency matrix of $O_{d+1}$. 
The action of $\A$ on $\C^{{\Omega\choose d}}$ is given by 
\begin{align}\label{A}
    \A x= \sum_{x\cap y=\emptyset} y \qquad
    \qquad \hbox{for all $x\in {\Omega\choose d}$},
\end{align}
where the sum is over all $y\in  {\Omega\choose d}$ with $x\cap y=\emptyset$. 
Pick any $x_0\in {\Omega\choose d}$.  Let $\A^*(x_0)$ denote the dual adjacency matrix of $O_{d+1}$ with respect to $x_0$. Recall from \cite{BannaiIto1984,DRG_book1989} that the action of $\A^*(x_0)$ on $\C^{{\Omega\choose d}}$ is given by 
\begin{align}\label{A^*}
    \A^*(x_0)x=
\left(
2d-\frac{2d+1}{d+1}(|x_0\setminus x|+|x\setminus x_0|)
\right)x 
    \qquad \hbox{for all $x\in {\Omega\choose d}$}.
\end{align}   
Let $\T(x_0)$ denote the Terwilliger algebra of $O_{d+1}$ with respect to $x_0$. 
As vector spaces, the $\BI$-module $\C^{2^\Omega}(x_0)(1)$ coincides with the standard $\T(x_0)$-module $\C^{{\Omega\choose d}}$. 
Building on this identification, our third main result gives a combinatorial realization of Theorems \ref{thm:V(1)} and \ref{thm:V(1)&LT}:

\begin{thm}\label{thm:BI->T} 
Let $d\geq 1$ be an integer and let $\Omega$ be a $(2d+1)$-element set.  
For any $x_0\in {\Omega \choose d}$ the following statements hold:
\begin{enumerate}
\item There exists a unique algebra homomorphism $\BI\to \T(x_0)$ given by
\begin{eqnarray*}
X &\mapsto& \A,
\\
Y&\mapsto& \frac{d+1}{2d+1} \A^*(x_0)+\frac{1}{2(2d+1)},
\\
Z&\mapsto&\frac{d+1}{2d+1}\{\A,\A^*(x_0)\}+\frac{1}{2d+1}\A.
\end{eqnarray*}
In particular the homomorphism is surjective.

\item The elements $X,Y,Z$ act on each irreducible $\T(x_0)$-module as a Leonard triple.
\end{enumerate}
\end{thm}

The paper is organized as follows: Section \ref{s:UZ/2Z_presentation} provides a presentation of 
the algebra $\U_{\Z/2\Z}$.
Section \ref{s:proofV(1)} contains the proof of Theorem \ref{thm:V(1)}.
In Section \ref{s:UZ/2Zmodule}, we classify finite-dimensional irreducible $\U_{\Z/2\Z}$-modules and show that every finite-dimensional $\U_{\Z/2\Z}$-module is completely reducible. Section \ref{s:BImodule} provides the necessary background on finite-dimensional irreducible $\BI$-modules. 
Section \ref{s:LT&BI} characterizes when the action of $X,Y,Z$ on a finite-dimensional irreducible $\BI$-module forms a Leonard triple. 
Section \ref{s:proofV(1)&LT} provides the proof of Theorem \ref{thm:V(1)&LT}.
The argument for Theorem \ref{thm:BI->T} is presented in Section \ref{s:proofBI->T}. 
Finally, we present the Clebsch--Gordan rule for $\U_{\Z/2\Z}$,  which provides the decomposition formula for finite-dimensional irreducible $\U_{\Z/2\Z}^{\otimes 2}$-modules into direct sums of irreducible $\U_{\Z/2\Z}$-modules.


\section{A presentation for $\U_{\Z/2\Z}$}
\label{s:UZ/2Z_presentation}

Recall that $\mathfrak{sl}_2$ is a three-dimensional Lie algebra over $\C$ with basis $E,F,H$. By (\ref{EFH_matrices}) the matrices $E,F,H$ satisfy 
\begin{align}
[H,E]&=2E, 
\label{[HE]}
\\
[H,F]&=-2F,
\label{[HF]}
\\ 
[E,F]&=H.
\label{[EF]}
\end{align}
Hence $\U$ is an algebra over $\C$ generated by $E,F,H$ subject to the relations (\ref{[HE]})--(\ref{[EF]}).

\begin{lem}\label{lem:Ubasis}
The elements  
\begin{align}\label{e:Ubasis}
 E^iF^j H^k
 \qquad 
 \hbox{for all $i,j,k\in \N$}
\end{align}
are a basis for $\U$.
\end{lem}
\begin{proof}
By Poincar\'{e}--Birkhoff--Witt theorem the lemma follows.
\end{proof}

\begin{lem}\label{lem:UZ/2Z_basis}
The elements  
\begin{align}\label{e:UZ/2Z_basis}
 E^iF^j H^k\rho^h \qquad 
 \hbox{ for all $i,j,k\in \N$ and $h\in\{0,1\}$}
\end{align}
are a basis for $\U_{\Z/2\Z}$.
\end{lem}
\begin{proof}
Immediate from Lemma \ref{lem:Ubasis} and the definition of $\U_{\Z/2\Z}$.
\end{proof}

\begin{thm}\label{thm:UZ/2Z_presentation}
The algebra $U(\mathfrak{sl}_2)_{\Z/2\Z}$ has a presentation generated by $E,F,H,\rho$ subject to the relations {\rm(\ref{[HE]})--(\ref{[EF]})} and
\begin{align}
    \rho H &=-H\rho,
      \label{UZ/2Z:Hrho}\\
    \rho E &= F\rho,
    \label{UZ/2Z:rhoE}\\
    \rho^2 &=1.
    \label{UZ/2Z:rho^2}
    \end{align}
\end{thm}
\begin{proof} 
Since $\U\subseteq \U_{\Z/2\Z}$ the relations (\ref{[HE]})--(\ref{[EF]}) hold in $\U_{\Z/2\Z}$. 
By (\ref{skew:product}) and (\ref{rho}), the relations (\ref{UZ/2Z:Hrho})--(\ref{UZ/2Z:rho^2}) hold in $\U_{\Z/2\Z}$. 
Let $\U_{\Z/2\Z}'$ denote the algebra over $\C$ generated by the symbols $E,F,H,\rho$ subject to the relations (\ref{[HE]})--(\ref{[EF]}) and (\ref{UZ/2Z:Hrho})--(\ref{UZ/2Z:rho^2}).  
By construction, there exists a unique algebra homomorphism $\Phi:\U_{\Z/2\Z}'\to \U_{\Z/2\Z}$ that sends 
\begin{eqnarray*}
(E,F,H,\rho) &\mapsto & (E,F,H,\rho).
\end{eqnarray*}
Combined with Lemma \ref{lem:UZ/2Z_basis} the elements (\ref{e:UZ/2Z_basis}) of $\U_{\Z/2\Z}'$ are linearly independent. 
To see that $\Phi$ is an isomorphism, it remains to verify that the elements (\ref{e:UZ/2Z_basis}) span $\U_{\Z/2\Z}'$.

Since the relations (\ref{[HE]})--(\ref{[EF]}) hold in $\U_{\Z/2\Z}'$, there exists a unique algebra homomorphism $\Psi:\U\to \U_{\Z/2\Z}'$ that sends 
\begin{eqnarray*}
(E,F,H) &\mapsto & (E,F,H).
\end{eqnarray*}
Since the elements (\ref{e:Ubasis}) of $\U_{\Z/2\Z}'$ are linearly independent, it follows from Lemma \ref{lem:Ubasis} that $\Psi$ is injective. Hence $\U$ can be considered as a subalgebra of $\U_{\Z/2\Z}'$. 
Let $\mathcal I$ 
denote the subspace of $\U_{\Z/2\Z}'$ spanned by  
(\ref{e:UZ/2Z_basis}); equivalently
\begin{gather*}
\mathcal I=\U+\U\rho.
\end{gather*}
Using (\ref{UZ/2Z:Hrho})--(\ref{UZ/2Z:rho^2}) yields that $\mathcal I$ is closed under right multiplication by  $E,F,H,\rho$. 
Since the algebra $\U_{\Z/2\Z}'$ is generated by these elements, it follows that $\mathcal I$ is a right ideal of $\U_{\Z/2\Z}'$. Finally, since $1\in \mathcal I$ we conclude that $\mathcal I=\U_{\Z/2\Z}'$. The result follows.
\end{proof}


\section{Proof for Theorem \ref{thm:V(1)}} 
\label{s:proofV(1)}

Recall $\Delta(H)$ and $\Delta(\rho)$ from \eqref{DeltaH} and \eqref{Deltarho}.

\begin{lem}
\label{prelem:V(theta)}
The following relations hold in $\U_{\Z/2\Z}^{\otimes 2}$:
\begin{enumerate}
\item $[\Delta(H),1\otimes H]=0$.

\item $[\Delta(H),H\otimes 1]=0$.

\item $[\Delta(H),1\otimes E]=2(1\otimes E)$.

\item $[\Delta(H),E\otimes 1]=2(E\otimes 1)$.

\item $\{\Delta(H),\Delta(\rho)\}=0$.
\end{enumerate}
\end{lem}
\begin{proof}
(i), (ii): Both $H\otimes 1$ and $1\otimes H$ commute with $\Delta(H)$. 

(iii), (iv): These are obtained by evaluating the commutators using \eqref{[HE]}. 

(v): This follows by evaluating the anti-commutator using  \eqref{UZ/2Z:Hrho}.
\end{proof}

Recall the notation $V(\theta)$ from (\ref{V(theta)}) for any $\U_{\Z/2\Z}^{\otimes 2}$-module $V$ and any $\theta\in \C$.

\begin{lem}
\label{lem:V(theta)}
Let $V$ denote a $\U_{\Z/2\Z}^{\otimes 2}$-module. For any $\theta\in \C$ the following relations hold:
\begin{enumerate}
\item $(H\otimes 1) V(\theta)\subseteq V(\theta)$.

\item $(1\otimes H) V(\theta)\subseteq V(\theta)$.

\item $(1\otimes E) V(\theta)\subseteq V(\theta+2)$.

\item $(E\otimes 1) V(\theta)\subseteq V(\theta+2)$.

\item $\Delta(\rho) V(\theta)\subseteq V(-\theta)$.
\end{enumerate}
\end{lem}
\begin{proof}
Immediate from Lemma \ref{prelem:V(theta)}.
\end{proof}

The distinguished central element
\begin{gather}
\label{Lambda}
\Lambda = EF + FE + \frac{H^2}{2}
\end{gather}
is called the {\it Casimir element} of $\U$. 
As in Section \ref{s:skew}, the Lie algebra automorphism $\rho$ of $\mathfrak{sl}_2$ can be regarded as an algebra automorphism of $\U$.

\begin{lem} 
\label{lem:Lambda}
\begin{enumerate}
\item The algebra automorphism $\rho$ of $\U$ fixes $\Lambda$.

\item The element $\Lambda$ is central in $\U_{\Z/2\Z}$.
\end{enumerate}
\end{lem}
\begin{proof}
(i): Apply $\rho$ to both sides of \eqref{Lambda} and evaluate the resulting equation using \eqref{rho}.

(ii): By \eqref{skew:product} the product $\rho\cdot \Lambda$ is equal to $\rho(\Lambda)\cdot \rho$ in $\U_{\Z/2\Z}$.
By (i) this implies that $\Lambda$ commutes with $\rho$. 
\end{proof}

\begin{proof}[Proof of Theorem \ref{thm:V(1)}]
For convenience, we temporarily set $X,Y,Z$ to denote the right-hand sides of (\ref{X:V(1)})--(\ref{Z:V(1)}) respectively. 
Lemma \ref{lem:V(theta)} implies that $V(1)$ is invariant under $X,Y,Z$. 
By Lemma \ref{lem:Lambda}(ii), the Casimir element $\Lambda$ of $\U$ is central in $\U_{\Z/2\Z}$.
It therefore suffices to verify that the following relations hold on $V(1)$:
\begin{align}
\{X,Y\} &=Z,
\label{V(1):BI-1}
\\
\{Y,Z\} &= X,
\label{V(1):BI-2}
\\
\{Z,X\} &=Y+\Lambda\otimes 1-1\otimes \Lambda.
\label{V(1):BI-3}
\end{align}

Using \eqref{Deltarho} and \eqref{UZ/2Z:Hrho}  a direct calculation shows that 
\begin{align*}
\{(E\otimes 1)\Delta(\rho),1\otimes H\}&=0,
\\
\{(1\otimes E)\Delta(\rho),H\otimes 1\}&=0.
\end{align*}
Using \eqref{Deltarho}, \eqref{[HE]} and  \eqref{UZ/2Z:Hrho} a direct calculation shows that 
\begin{align*}
\{(E\otimes 1)\Delta(\rho),H\otimes 1\}&=2(E\otimes 1)\Delta(\rho),
\\
\{(1\otimes E)\Delta(\rho),1\otimes H\}
&=2(1\otimes E)\Delta(\rho).
\end{align*}
With these four identities, it is straightforward to verify the relations \eqref{V(1):BI-1} and \eqref{V(1):BI-2}.
Applying \eqref{DeltaE}, \eqref{Deltarho}, \eqref{UZ/2Z:rhoE} and \eqref{UZ/2Z:rho^2} a direct computation yields
\begin{align*}
\{(E\otimes 1)\Delta(\rho),\Delta(E\rho)\}
&=E\otimes F+F\otimes E+2(EF\otimes 1),
\\
\{ (1\otimes E)\Delta(\rho), \Delta(E\rho)\}
&=E\otimes F+F\otimes E+2(1\otimes EF).
\end{align*}
These identities imply that 
$$
\{Z,X\}-Y=
\left(2EF-\frac{H}{2}\right)\otimes 1
-
1\otimes\left(2EF-\frac{H}{2}\right).
$$ 
By \eqref{[EF]} and \eqref{Lambda} the element $2EF-\frac{H}{2}$ can be written as $\Lambda-\frac{H(H-1)}{2}$. Hence 
$\{Z,X\}-Y$ is equal to $\Lambda\otimes1-1\otimes \Lambda$
plus $\frac{1}{2}$ times  
\begin{gather*}
1\otimes H(H-1)-H(H-1)\otimes 1.
\end{gather*}
Using (\ref{DeltaH}) the above element factors as
$(1\otimes H-H\otimes 1)(\Delta(H)-1)$, which vanishes on 
$V(1)$ by \eqref{V(theta)}. 
Therefore \eqref{V(1):BI-3} holds on $V(1)$. The result follows.
\end{proof}

Let $\kappa,\lambda,\mu$ denote the defining central elements (\ref{BI:centers}) of $\BI$.  Explicitly,
\begin{align}
\kappa&=\{X,Y\}-Z,  \label{BI:center1}
\\
\lambda&=\{Y,Z\}-X, \label{BI:center2}
\\
\mu&=\{Z,X\}-Y.  \label{BI:center3}
\end{align}
In addition, the distinguished central element
$$
X^2+Y^2+Z^2
$$ 
is called the {\it Casimir element} of $\BI$ \cite[Equation (6)]{R&BI2015}.

\begin{lem}
\label{lem:klm&Casimir_V(1)}
For any $\U_{\Z/2\Z}^{\otimes 2}$-module $V$ the following statements hold:
\begin{enumerate}
\item $\kappa,\lambda,\mu$ act on the $\BI$-module $V(1)$ as $0$, $0$, $\Lambda\otimes 1-1\otimes \Lambda$ respectively.

\item $X^2+Y^2+Z^2$ acts on the $\BI$-module $V(1)$ as 
$\Lambda\otimes 1+1\otimes \Lambda+\frac{3}{4}$.
\end{enumerate}
\end{lem}
\begin{proof}
(i): The statement was shown in the proof of Theorem \ref{thm:V(1)}.

(ii): 
Recall the actions of $X,Y,Z$ on the $\BI$-module $V(1)$ from \eqref{X:V(1)}--\eqref{Z:V(1)}.  By the formula for the sum of squares we have 
\begin{align*}
Y^2 &=\frac{H^2\otimes 1+1\otimes H^2}{4}
-\frac{H\otimes H}{2}.
\end{align*}
Using (\ref{DeltaE}), (\ref{DeltaF}), (\ref{UZ/2Z:rhoE}) and (\ref{UZ/2Z:rho^2}) a direct calculation shows that 
\begin{align*}
X^2&=
EF\otimes 1+E\otimes F+F\otimes E+1\otimes EF,
\\
Z^2&=EF\otimes 1-E\otimes F-F\otimes E+1\otimes EF.
\end{align*}
Summing these identities, the Casimir element $X^2+Y^2+Z^2$ of $\BI$ acts on $V(1)$ as 
\begin{gather}
\label{X^2+Y^2+Z^2:V(1)}
\left(
2EF+\frac{H^2}{4}
\right)\otimes 1
+
1\otimes \left(
2EF+\frac{H^2}{4}
\right)
-\frac{H\otimes H}{2}.
\end{gather}

By \eqref{[EF]} and \eqref{Lambda} the element $2EF+\frac{H^2}{4}$ can be written as $\Lambda+H-\frac{H^2}{4}$. Substituting this into \eqref{X^2+Y^2+Z^2:V(1)} and rewriting in terms of $\Delta(H)$ yields that \eqref{X^2+Y^2+Z^2:V(1)} is equal to $\Lambda\otimes 1+1\otimes \Lambda$ plus 
$$
\Delta(H) \left( 1-\frac{\Delta(H)}4 \right).
$$
Since $\Delta(H)$ acts as $1$ on $V(1)$ by \eqref{V(theta)}, the statement (ii) follows. 
\end{proof}

\section{Finite-dimensional $\U_{\Z/2\Z}$-modules} 
\label{s:UZ/2Zmodule}

We begin this section with some standard facts on finite-dimensional $\U$-modules.
For each $n\in \N$ there exists an $(n+1)$-dimensional $\U$-module $L_n$ with basis $\{ v_i^{(n)} \}_{i=0}^n$ satisfying 
\begin{align}
E v_i^{(n)}&=i v_{i-1}^{(n)}  \qquad (1\leq i\leq n),
 \qquad Ev_0^{(n)}=0, 
\label{Ln:E}
\\
F v_i^{(n)}&=(n-i)  v_{i+1}^{(n)} \qquad  (0\leq i\leq n-1),
\qquad F v_n^{(n)}=0,
\label{Ln:F}
\\
H v_i^{(n)}&=(n-2i) v_i^{(n)} \qquad  (0\leq i\leq n).
\label{Ln:H}
\end{align}
The Casimir element $\Lambda$ acts on $L_n$ as scalar multiplication by $\frac{n(n+2)}{2}$. 
The $\U$-module $L_n$ is irreducible.

\begin{lem} 
\label{lem:Ln} 
For each $n\in \N$ every $(n+1)$-dimensional irreducible $\U$-module is isomorphic to $L_{n}$.
\end{lem}
\begin{proof}
See \cite[Section V.4]{kassel} for example.
\end{proof}

Recall that a module over an algebra is called {\it completely reducible} if it is equal to a (direct) sum of irreducible submodules \cite[Section XVII.2]{LangAlgebra}.

\begin{lem}
\label{lem:Umodule_semisimple} 
Every finite-dimensional $\U$-module is completely reducible.
\end{lem}
\begin{proof}
See \cite[Theorem V.4.6]{kassel}  for example.
\end{proof}

Applying Theorem \ref{thm:UZ/2Z_presentation}, the $\U$-module $L_n$ ($n\in \N$) can be extended to a $\U_{\Z/2\Z}$-module by defining
\begin{gather}
\label{L+&L-:rho}
\rho\, v_i^{(n)} =\pm v_{n-i}^{(n)}
\qquad (0\leq i\leq n).
\end{gather}
The choice of sign gives the two distinct $\U_{\Z/2\Z}$-modules, denoted by $L_n^+$ and $L_n^-$, respectively.

\begin{lem}
\label{lem:Ln+&Ln-}
For any $n\in \N$ the following statements hold:
\begin{enumerate}
\item 
The $\U_{\Z/2\Z}$-modules $L_n^+$ and $L_n^{-}$ are irreducible.

\item The $\U_{\Z/2\Z}$-modules $L_n^+$ and $L_n^{-}$ are not isomorphic.
\end{enumerate}
\end{lem}
\begin{proof}
(i): By construction, the restriction of the $\U_{\Z/2\Z}$-modules $L_n^\pm$ to $\U$ is the irreducible $\U$-module $L_n$. Hence (i) follows.

(ii): Suppose, for contradiction, that there exists a $\U_{\Z/2\Z}$-module isomorphism $f: L_n^+ \to L_n^-$. Then $f$ is a $\U$-module automorphism of $L_n$. By Schur's lemma $f$ must be a nonzero scalar multiple of the identity map on $L_n$.
However, by \eqref{L+&L-:rho} any nonzero scalar multiple of the identity cannot commute with $\rho$, a contradiction.
 \end{proof}

\begin{lem}
\label{lem1:W&rhoW}
Suppose that $V$ is a $\U_{\Z/2\Z}$-module and $W$ is a $\U$-submodule of $V$. 
Then the following statements hold:
\begin{enumerate}
\item $\rho W$ is a $\U$-submodule of $V$.

\item  $W+\rho W$ is a $\U_{\Z/2\Z}$-submodule of $V$.
\end{enumerate}
\end{lem}
\begin{proof}
(i): Using \eqref{UZ/2Z:Hrho}--\eqref{UZ/2Z:rho^2} it is routine to verify that $\rho W$ is invariant under $E,F,H$. Since the algebra $\U$ is generated by these elements, the statement (i) follows.

(ii): By (i) the subspace $W+\rho W$ of $V$ is a $\U$-module.
In view of (\ref{UZ/2Z:rho^2}) the $\U$-module $W+\rho W$ is also invariant under $\rho$. The statement (ii) follows.
\end{proof}

\begin{lem}
\label{lem2:W&rhoW}
Let $V$ denote a $\U_{\Z/2\Z}$-module. Suppose that $W$ is a $\U$-submodule of $V$ isomorphic to $L_n$ for some $n\in \N$. 
Then  the following statements are true:
\begin{enumerate}
\item The $\U$-module $\rho W$ is isomorphic to $L_n$.

\item Suppose that $W=\rho W$. Then
the $\U_{\Z/2\Z}$-module $W$ is isomorphic to either $L_n^+$ or $L_n^-$.

\item Suppose that $W\not= \rho W$.
Then the $\U_{\Z/2\Z}$-module $W+ \rho W$ is isomorphic to $L_n^+\oplus L_n^-$. 

\end{enumerate}
\end{lem}
\begin{proof}
Since the $\U$-module $W$ is isomorphic to $L_n$, there is a $\U$-module isomorphism $f:L_n\to W$. Define a linear map $g:L_n\to \rho W$ by
\begin{eqnarray}\label{g:Ln->rhoW}
v_i^{(n)} 
&\mapsto & \rho f (v_{n-i}^{(n)})
\qquad (0\leq i\leq n).
\end{eqnarray}
(i): By Lemma \ref{lem1:W&rhoW}(i) the space $\rho W$ is a $\U$-submodule of $V$. Since $f$ is a $\U$-module homomorphism, it is routine to verify that $g$ is also a $\U$-module homomorphism by using \eqref{UZ/2Z:Hrho}, \eqref{UZ/2Z:rhoE} and \eqref{Ln:E}--\eqref{Ln:H}.  
Since $f$ is a linear isomorphism and $\rho$ is invertible in $\U_{\mathbb{Z}/2\mathbb{Z}}$ by (\ref{UZ/2Z:rho^2}), the vectors $\{\rho f(v_{n-i}^{(n)})\}_{i=0}^n$ form a basis of $\rho W$. Therefore $g$ is a $\U$-module isomorphism. The statement (i) follows.

(ii): Since $W=\rho W$ the composition $f^{-1}\circ g$ is a $\U$-module automorphism of $L_n$.  
By Schur's lemma there exists a nonzero scalar $\e\in\C$ such that $f^{-1}\circ g=\e\cdot 1$. 
Consequently $g=\e\cdot f$. Evaluating both sides at $v_{n-i}^{(n)}$  and applying (\ref{g:Ln->rhoW}) gives
\begin{align}\label{e1:W&rhoW}
\rho f(v_i^{(n)})=\e f(v_{n-i}^{(n)})
\qquad
(0\leq i\leq n).
\end{align}

Applying $\rho$ to both sides of (\ref{e1:W&rhoW}), the right-hand side is evaluated using (\ref{e1:W&rhoW}) and the left-hand side is evaluated by (\ref{UZ/2Z:rho^2}). It follows that
$$
f(v_i^{(n)})=\e^2 f(v_i^{(n)})
\qquad
(0\leq i\leq n).
$$
Since the vectors $\{f(v_i^{(n)})\}_{i=0}^n$ are nonzero, it follows that $\e\in\{\pm1\}$. Comparing (\ref{L+&L-:rho}) with  (\ref{e1:W&rhoW}) implies that if $\e=1$ then $f$ is a $\U_{\Z/2\Z}$-module isomorphism from $L_n^+$ to $W$; if $\e=-1$ then $f$ is a $\U_{\Z/2\Z}$-module isomorphism from $L_n^-$ to $W$. 
The statement (ii) follows.

(iii): 
By Lemma \ref{lem1:W&rhoW}(ii) the space $W+ \rho W$ is a $\U_{\Z/2\Z}$-submodule of $V$. 
Define a map $h^+:L_n^+\to  W+ \rho W$ by 
\begin{align*}
h^+(v)=f(v)+g(v) \qquad 
\hbox{for all $v\in L_n^+$}.
\end{align*}
Define a map $h^-:L_n^-\to W+\rho W$ by 
\begin{align*}
h^-(v)=f(v)-g(v) \qquad 
\hbox{for all $v\in L_n^-$}.
\end{align*}
Since $f$ and $g$ are $\U$-module homomorphisms, the maps $h^+$ and $h^-$ are also $\U$-module homomorphisms.
Using (\ref{UZ/2Z:rho^2}), (\ref{L+&L-:rho}) and (\ref{g:Ln->rhoW}) yields that 
\begin{align*}
\rho
\left(h^+(v_i^{(n)})
\right) = h^+
\left(\rho\, v_i^{(n)}
\right) 
\qquad (0\leq i\leq n),
\\
\rho\left(h^-(v_i^{(n)})
\right) = h^-
\left(\rho\, v_i^{(n)}\right)
\qquad (0\leq i\leq n).
\end{align*}
Since the vectors $\{v_i^{(n)}\}_{i=0}^n$ span $L_n^+$ and $L_n^-$, it follows that $h^+$ and $h^-$ are $\U_{\Z/2\Z}$-module homomorphisms.

Since $f:L_n\to W$ and $g:L_n\to \rho W$ are $\U$-module isomorphisms, the $\U$-modules $W$ and $\rho W$ are irreducible. Since $W\not=\rho W$ by assumption, it follows that $W\cap \rho W=\{0\}$. 
Combined with the injectivity of $f$ and $g$, the kernels of $h^+$ and $h^-$ are trivial.
Thus the images of $h^+$ and $h^-$ are two $\U_{\Z/2\Z}$-submodules of $W+\rho W$, isomorphic to $L_n^+$ and $L_n^-$, respectively. Finally, by Lemma \ref{lem:Ln+&Ln-}(ii) and since $\dim (W+\rho W)$ and $\dim L_n^+ +\dim L_n^-$ are equal to $2(n+1)$, the statement (iii) follows.
\end{proof}

We are now in a position to provide a complete classification of finite-dimensional irreducible $\U_{\Z/2\Z}$-modules and establish the complete reducibility of finite-dimensional $\U_{\Z/2\Z}$-modules.

\begin{thm} 
\label{thm:UZ/2Zmodule_classification}
For any $n\in \N$ every $(n+1)$-dimensional irreducible $\U_{\Z/2\Z}$-module is isomorphic to either $L_{n}^+$ or $L_{n}^-$.
\end{thm}
\begin{proof}
Let $n\in \N$ be given. Suppose that $V$ is an $(n+1)$-dimensional irreducible $\U_{\Z/2\Z}$-module. 
Then $V$ contains a finite-dimensional irreducible $\U$-submodule $W$. By Lemma \ref{lem:Ln} the $\U$-module $W$ is isomorphic to $L_m$ for some $m\in \N$. 

Suppose that $W\not= \rho W$. 
By Lemma \ref{lem2:W&rhoW}(iii) the space $W+ \rho W$ is a $\U_{\Z/2\Z}$-submodule of $V$ isomorphic to $L_m^+\oplus L_m^-$. This contradicts the irreducibility of $V$ as a $\U_{\Z/2\Z}$-module. Hence $W=\rho W$. 
By Lemma \ref{lem2:W&rhoW}(ii) the space $W$ is a $\U_{\Z/2\Z}$-submodule of $V$ isomorphic to either $L_m^+$ or $L_m^-$. Since the $\U_{\Z/2\Z}$-module $V$ is irreducible, it follows that $V=W$ and $m=n$. The result follows.
\end{proof}

\begin{thm}\label{thm:UZ/2Zmodule_semisimple}
Every finite-dimensional $\U_{\Z/2\Z}$-module is completely reducible.
\end{thm}
\begin{proof} 
Suppose that $V$ is a finite-dimensional $\U_{\Z/2\Z}$-module.
By Lemma \ref{lem:Umodule_semisimple} the $\U$-module $V$ is a sum of irreducible $\U$-submodules of $V$.
It follows that
\begin{align*}
V=\sum_{\hbox{\tiny irreducible $\U$-submodules $W$ of $V$}}
W+\rho W, 
\end{align*}
where the sum is over all irreducible $\U$-submodules $W$ of $V$.

Let $W$ be any irreducible $\U$-submodule of $V$. By Lemma \ref{lem:Ln} there exists an $n\in \N$ such that the $\U$-module $W$ is isomorphic to $L_n$. 
By Lemma \ref{lem2:W&rhoW} the space $W+\rho W$ is a $\U_{\Z/2\Z}$-submodule of $V$ isomorphic to $L_n^+$, $L_n^-$ or $L_n^+\oplus L_n^-$. 
Since the $\U_{\Z/2\Z}$-modules $L_n^+$ and $L_n^-$ are irreducible by Lemma \ref{lem:Ln+&Ln-}(i), 
it follows that $V$ is equal to a sum of irreducible $\U_{\Z/2\Z}$-submodules of $V$. The result follows.
\end{proof}

\section{Finite-dimensional irreducible $\BI$-modules} \label{s:BImodule}

Recall the central elements $\kappa,\lambda,\mu$ of $\BI$ defined in \eqref{BI:center1}--\eqref{BI:center3}.

\begin{prop}
[Proposition 2.6, \cite{Huang:BImodule}]
\label{prop:BImodule_odddim}
For any $a,b,c\in \C$ and any even integer $n\geq 0$, there exists an $(n+1)$-dimensional $\BI$-module $O_n(a,b,c)$ satisfying the following conditions:
\begin{enumerate}
\item  There exists a basis $\{u_i\}_{i=0}^n$ for $O_n(a,b,c)$ such that 
\begin{align*}
(X-\theta_i)u_i&=u_{i+1} \quad (0\leq i\leq n-1),
\qquad 
(X-\theta_n)u_n=0,
\\
(Y-\theta_i^*)u_i&=\varphi_i u_{i-1} \quad (1\leq i\leq n),
\qquad 
(Y-\theta_0^*)u_0=0,
\end{align*}
where
\begin{align*}
\theta_i&=
\frac{(-1)^i(2a-n+2i)}{2} 
\qquad(0\leq i\leq n),
\\
\theta^*_i&=
\frac{(-1)^i(2b-n+2i)}{2} 
\qquad(0\leq i\leq n),
\\
\varphi_i&=
\left\{
\begin{array}{ll}
\displaystyle \frac{i(n+1-2i-2a-2b-2c)}2
\quad 
&\hbox{if $i$ is even},
\\
\displaystyle \frac{(i-n-1)(n+1-2i-2a-2b+2c)}2
\quad &\hbox{if $i$ is odd}
\end{array}
\right.
\qquad(1\leq i\leq n).
\end{align*}

\item The elements $\kappa,\lambda,\mu$ act on $O_n(a,b,c)$ as scalar multiplication by
\begin{align*}
&2ab-c(n+1),
\\
&2bc-a(n+1),
\\
&2ca-b(n+1),
\end{align*}
respectively.
\end{enumerate}
\end{prop}

\begin{lem}\label{lem:BImodule_odddim_trace}
For any $a,b,c\in \C$ and any even integer $n\geq 0$, the traces of $X,Y,Z$ on $O_n(a,b,c)$ are equal to $a,b,c$ respectively.
\end{lem}
\begin{proof}
By Proposition~\ref{prop:BImodule_odddim} the traces of $X,Y,XY$ on $O_n(a,b,c)$ are
\begin{gather*}
\sum_{i=0}^n \theta_i = a,
\qquad 
\sum_{i=0}^n \theta_i^* = b,
\qquad 
\sum_{i=0}^n \theta_i\theta_i^* + \sum_{i=1}^n \varphi_i
=ab(n+1)-\frac{cn(n+2)}{2},
\end{gather*}
respectively. By \eqref{BI:center1} the element $Z=\{X,Y\}-\kappa$.
Taking the trace on $O_n(a,b,c)$ shows that the trace of $Z$ on $O_n(a,b,c)$ is $c$. The lemma follows. 
\end{proof}

\begin{thm}[Theorem 2.8, \cite{Huang:BImodule}]
\label{thm:BImodule_odddim_classification}
Let $n\geq 0$ be an even integer. Suppose that $V$ is an $(n+1)$-dimensional irreducible $\BI$-module. Then there exist unique scalars $a,b,c\in\C$ such that the $\BI$-module $O_n(a,b,c)$ is isomorphic to $V$. Moreover, for any $a,b,c\in \C$ the $\BI$-module $O_n(a,b,c)$ is irreducible if and only if 
$$
a+b+c,a-b-c,-a+b-c,-a-b+c\notin\left\{
\frac{n+1}{2}-i \,\Bigg|\, i=2,4,\dots,n\right\}.
$$
\end{thm}

Applying Lemma \ref{lem:BImodule_odddim_trace}, we obtain the following criterion for determining the parameters $a,b,c$ of $O_n(a,b,c)$ to which a given odd-dimensional irreducible $\BI$-module is isomorphic.

\begin{lem}
\label{lem:BImodule_odddim_criterion}
Let $n\geq 0$ be an even integer. Suppose that $V$ is an $(n+1)$-dimensional irreducible $\BI$-module. For any $a,b,c\in \C$, the $\BI$-module $O_n(a,b,c)$ is isomorphic to $V$ if and only if the traces of $X,Y,Z$ on $V$ are equal to $a,b,c$    respectively.
\end{lem}

We now turn our attention to even-dimensional irreducible $\BI$-modules.

\begin{prop}
[Proposition 2.4, \cite{Huang:BImodule}]
\label{prop:BImodule_evendim}
For any $a,b,c\in \C$ and any odd integer $n\geq 1$, there exists an $(n+1)$-dimensional $\BI$-module $E_n(a,b,c)$ satisfying the following conditions:
\begin{enumerate}
\item  There exists a basis $\{u_i\}_{i=0}^n$ for $E_n(a,b,c)$ such that 
\begin{align*}
(X-\theta_i)u_i&=u_{i+1} \qquad (0\leq i\leq n-1),
\qquad 
(X-\theta_n)u_n=0,
\\
(Y-\theta_i^*)u_i&=\varphi_i u_{i-1} \qquad (1\leq i\leq n),
\qquad 
(Y-\theta_0^*)u_0=0,
\end{align*}
where
\begin{align*}
\theta_i&=
\frac{(-1)^i(2a-n+2i)}2 \qquad(0\leq i\leq n),
\\
\theta^*_i&=
\frac{(-1)^i(2b-n+2i)}2 \qquad(0\leq i\leq n),
\\
\varphi_i&=
\left\{
\begin{array}{ll}
i(n-i+1)
\quad 
&\hbox{if $i$ is even},
\\
\displaystyle c^2-
\left(\frac{2a+2b-n+2i-1}{2}\right)^2
\quad &\hbox{if $i$ is odd}
\end{array}
\right.
\qquad(1\leq i\leq n).
\end{align*}

\item The elements $\kappa,\lambda,\mu$ act on $E_n(a,b,c)$ as scalar multiplication by 
\begin{align*}
&c^2-a^2-b^2+\left(\frac{n+1}{2}\right)^2,
\\
&a^2-b^2-c^2+\left(\frac{n+1}{2}\right)^2,
\\
&b^2-c^2-a^2+\left(\frac{n+1}{2}\right)^2,
\end{align*}
respectively.
\end{enumerate}
\end{prop}

\begin{lem}
\label{lem:BImodule_evendim_trace} 
For any $a,b,c\in \C$ and any odd integer $n\geq 1$, the traces of $X,Y,Z$ on $E_n(a,b,c)$ are equal to $-\frac{n+1}{2}$.
\end{lem}
\begin{proof}
By Proposition \ref{prop:BImodule_evendim} the traces of $X, Y,XY$ on $E_n(a,b,c)$ are equal to 
\begin{gather*}
\sum_{i=0}^n \theta_i=-\frac{n+1}{2}, \qquad 
\sum_{i=0}^n \theta_i^*=-\frac{n+1}{2},
\\
\sum_{i=0}^n \theta_i\theta_i^*+\sum_{i=1}^n \varphi_i
=\frac{(n+1)(4c^2-4a^2-4b^2+n^2+2n-1)}{8},
\end{gather*}
respectively.
By \eqref{BI:center1} the element $Z=\{X,Y\}-\kappa$.
Taking the trace on $E_n(a,b,c)$ shows that the trace of $Z$ on $E_n(a,b,c)$ is $-\frac{n+1}{2}$. The lemma follows. 
\end{proof}

The set $\{\pm 1\}$ is a group under multiplication. There is a natural $\{\pm 1\}^3$-action on $\C^3$ given by 
\begin{align*}
(a,b,c)^{(-1,1,1)}&=(-a,b,c),
\\
(a,b,c)^{(1,-1,1)}&=(a,-b,c),
\\
(a,b,c)^{(1,1,-1)}&=(a,b,-c)
\end{align*}
for all $(a,b,c)\in \C^3$. 
There exists a unique $\{\pm 1\}^2$-action on $\BI$ such that each $g\in \{\pm 1\}^2$ acts as the algebra automorphism of $\BI$ specified in the following table:
\begin{table}[H]
\centering
\extrarowheight=3pt
\begin{tabular}{c||rrr|rrr}
$u$  
&$X$ &$Y$ &$Z$ &$\kappa$ & $\lambda$ &$\mu$
\\
\midrule[1pt]
$(1,1)(u)$
& $X$    &$Y$ &$Z$ &$\kappa$ & $\lambda$ &$\mu$
\\
$(1,-1)(u)$
& $X$    &$-Y$ &$-Z$ &$-\kappa$ & $\lambda$ &$-\mu$
\\
$(-1,1)(u)$
& $-X$  &$Y$  &$-Z$&$-\kappa$ & $-\lambda$ &$\mu$
\\
$(-1,-1)(u)$
& $-X$  &$-Y$  &$Z$ &$\kappa$ & $-\lambda$ &$-\mu$
\end{tabular}
\caption{The $\{\pm 1\}^2$-action on $\BI$}
\label{t:BI&pm1^2}
\end{table}

Let $V$ be a module over an algebra. For any automorphism $\varepsilon$ of the algebra, we denote by $V^{\e}$ the module obtained from $V$ by twisting the module structure via $\e$.

\begin{thm}
[Theorem 2.5, \cite{Huang:BImodule}]
\label{thm:BImodule_evendim_classification}
Let $n\geq 1$ be an odd integer. Suppose that $V$ is an $(n+1)$-dimensional irreducible $\BI$-module. Then there exist scalars $a,b,c\in \C$ and a unique element $(\e,\e')\in \{\pm 1\}^2$ such that the $\BI$-module $E_n(a,b,c)^{(\e,\e')}$ is isomorphic to $V$, where the triple $(a,b,c)$ is unique up to the $\{\pm 1\}^3$-action on $\C^3$. Moreover, for any $a,b,c\in \C$ the $\BI$-module $E_n(a,b,c)$ is irreducible if and only if 
$$
a+b+c,-a+b+c,a-b+c,a+b-c\notin 
\left\{ \frac{n-1}{2}-i\,\Bigg|\,i=0,2,\ldots,d-1\right\}.
$$ 
\end{thm}

The following lemma gives a criterion for determining the parameters $\e,\e'$ and $a,b,c$ of $E_n(a,b,c)^{(\e,\e’)}$ to which a given even-dimensional irreducible $\BI$-module is isomorphic.

\begin{lem}
\label{lem:BImodule_evendim_criterion}
Let $n\geq 1$ be an odd integer. Assume that $V$ is an $(n+1)$-dimensional irreducible $\BI$-module. For any $a,b,c\in \C$ and any $(\e,\e')\in \{\pm 1\}^2$, the $\BI$-module $V$ is isomorphic to $E_n(a,b,c)^{(\e,\e')}$ if and only if the following conditions hold:
\begin{enumerate}
\item The traces of $X$ and $Y$ on the $\BI$-module $V^{(\e,\e')}$ are equal to $-\frac{n+1}{2}$.

\item The elements $\kappa,\lambda,\mu$ act on the $\BI$-module $V^{(\e,\e')}$ as scalar multiplication by 
\begin{align*}
&c^2-a^2-b^2+\left(\frac{n+1}{2}\right)^2,
\\
&a^2-b^2-c^2+\left(\frac{n+1}{2}\right)^2,
\\
&b^2-c^2-a^2+\left(\frac{n+1}{2}\right)^2,
\end{align*}
respectively.
\end{enumerate}
\end{lem}
\begin{proof}
($\Rightarrow$): Condition (i) follows immediately from Lemma \ref{lem:BImodule_evendim_trace} and Table \ref{t:BI&pm1^2}. 
Condition (ii) follows immediately from Proposition \ref{prop:BImodule_evendim}(ii).

($\Leftarrow$): 
By Theorem \ref{thm:BImodule_evendim_classification} there exist $a',b',c'\in \C$ and $(\epsilon,\epsilon')\in\{\pm1\}^2$ such that the $\BI$-module $V$ is isomorphic to $E_n(a',b',c')^{(\epsilon,\epsilon')}$. 
By Table \ref{t:BI&pm1^2} and Lemma  \ref{lem:BImodule_evendim_trace}, the traces of $X$ and $Y$ on $V$ are $-\epsilon \frac{n+1}{2}$ and $-\epsilon' \frac{n+1}{2}$, respectively. Condition (i) then forces $(\epsilon,\epsilon’)=(\e,\e’)$.

It follows that the $\BI$-module $V$ is isomorphic to $E_n(a',b',c')^{(\e,\e')}$; equivalently, the $\BI$-module $V^{(\e,\e')}$ is isomorphic to $E_n(a',b',c')$.
Using Proposition \ref{prop:BImodule_evendim} the elements $\kappa,\lambda,\mu$ act on $V^{(\e,\e')}$ as scalar multiplication by 
\begin{align*}
&c'^2-a'^2-b'^2+\left(\frac{n+1}{2}\right)^2,
\\
&a'^2-b'^2-c'^2+\left(\frac{n+1}{2}\right)^2,
\\
&b'^2-c'^2-a'^2+\left(\frac{n+1}{2}\right)^2,
\end{align*}
respectively. Comparing with condition (ii), this implies that $(a',b',c')$ lies in the $\{\pm 1\}^3$-orbit of $(a,b,c)$. By Theorem  \ref{thm:BImodule_evendim_classification} the $\BI$-module $E_n(a,b,c)^{(\e,\e')}$ is therefore isomorphic to $V$.
The lemma follows. 
\end{proof}


\section{Leonard triples on finite-dimensional irreducible $\BI$-modules} \label{s:LT&BI}

A square matrix is called {\it tridiagonal} if all nonzero entries lie on the main diagonal, the subdiagonal, or the superdiagonal. 
A tridiagonal matrix is called {\it irreducible} if every entry on the subdiagonal and every entry on the superdiagonal is nonzero.

\begin{defn} 
[Definition 1.2, \cite{cur2007-2}]
\label{defn:LT}
Let $V$ be a nonzero finite-dimensional vector space. An ordered triple of linear operators $L : V \to V$, $L^* : V \to V$, $L^\e : V \to V$ is called a {\it Leonard triple} if the following conditions hold:
\begin{enumerate}
\item There exists a basis for $V$ with respect to which the matrix representing $L$ is diagonal and the matrices representing $L^*$ and $L^\e$ are irreducible tridiagonal.

\item There exists a basis for $V$ with respect to which the matrix representing $L^*$ is diagonal and the matrices representing $L^\e$ and $L$ are irreducible tridiagonal.

\item There exists a basis for $V$ with respect to which the matrix representing $L^\e$ is diagonal and the matrices representing $L$ and $L^*$ are irreducible tridiagonal.
\end{enumerate}
\end{defn}

In this section, we derive necessary and sufficient conditions for $X,Y,Z$ to form Leonard triples on finite-dimensional irreducible $\BI$-modules.

\begin{lem}
\label{lem:BI&LT}
The following relations hold in $\BI$:
\begin{enumerate}
\item $X^2 Y+2 XYX+YX^2-Y=2\kappa X+\mu$.

\item $X^2 Z+2 XZX+ZX^2-Z=2\mu X+\kappa$.
\end{enumerate}
\end{lem}
\begin{proof}
Relation (i) follows by eliminating $Z$ from \eqref{BI:center3} using \eqref{BI:center1}.
Relation (ii) follows by eliminating $Y$ from \eqref{BI:center1} using \eqref{BI:center3}.
\end{proof}

\begin{lem}
\label{lem:On(a,b,c)&LT}
Let $a,b,c\in \C$ and $n\geq 0$ be an even integer.  
Suppose that the $\BI$-module $O_n(a,b,c)$ is irreducible. Then the following conditions are equivalent: 
\begin{enumerate}
\item There exists a basis for $O_n(a,b,c)$ with respect to which the matrix representing $X$ is diagonal and the matrices representing $Y$ and $Z$ are irreducible tridiagonal.

\item $X$ is diagonalizable on $O_n(a,b,c)$.

\item $a\notin\{\frac{n-1}{2}-i\,|\,i=0,1,\ldots,n-1\}$.
\end{enumerate}
\end{lem}
\begin{proof}
We use the notations from Proposition \ref{prop:BImodule_odddim} in the proof.

(i) $\Rightarrow$ (ii): Trivial.

(ii) $\Leftrightarrow$ (iii): With respect to the basis $\{u_i\}_{i=0}^n$ for $O_n(a,b,c)$, the matrix representing $X$ is lower bidiagonal, with diagonal entries $\{\theta_i\}_{i=0}^n$ and all subdiagonal entries equal to $1$. Consequently, the minimal polynomial of $X$ on $O_n(a,b,c)$ is 
$$
\prod_{i=0}^n (x-\theta_i).
$$
It follows that $X$ is diagonalizable if and only if the scalars $\{\theta_i\}_{i=0}
^n$ are pairwise distinct. The latter condition is equivalent to (iii).

(iii) $\Rightarrow$ (i): 
By (iii) all eigenvalues $\{\theta_i\}_{i=0}
^n$ of $X$ on $O_n(a,b,c)$ are simple. For $0\leq i\leq n$, let $v_i$ denote a  $\theta_i$-eigenvector of $X$ in $O_n(a,b,c)$. Then the vectors $\{v_i\}_{i=0}^n$ are a basis for $O_n(a,b,c)$. 
Set 
$$
\theta_{-1}=\frac{n}{2}-a+1,
\qquad 
\theta_{n+1}=-\frac{n}{2}-a-1.
$$
Applying both sides of Lemma \ref{lem:BI&LT}(i) to $v_i$ ($0\leq i\leq n$) shows that 
$$
(X-\theta_{i-1})(X-\theta_{i+1}) Y v_i
$$
is a scalar multiple of $v_i$. Since $\{\theta_i\}_{i=0}^n$ are pairwise distinct by (iii), $Yv_i$ is a linear combination of $v_{i-1},v_i,v_{i+1}$ for all $i=1,2,\ldots,n-1$. 
By construction, $\theta_{-1}$ and $\theta_{n+1}$ are not in $\{\theta_1,\theta_3,\ldots,\theta_{n-1}\}$. By (iii), $\theta_{-1}\notin \{\theta_2,\theta_4,\ldots,\theta_n\}$ and $\theta_{n+1}\notin \{\theta_0,\theta_2,\ldots,\theta_{n-2}\}$. 
It follows that $\theta_{-1}$ is not among the scalars $\{\theta_i\}_{i=1}^n$ and $\theta_{n+1}$ is not among $\{\theta_i\}_{i=0}^{n-1}$. Consequently, $Yv_0$ is a linear combination of $v_0$ and $v_1$, and $Yv_n$ is a linear combination of $v_{n-1}$ and $v_n$. Thus, the matrix representing $Y$ with respect to the basis $\{v_i\}_{i=0}^n$ for $O_n(a,b,c)$ is tridiagonal. Similarly, Lemma \ref{lem:BI&LT}(ii), together with condition (iii), implies the matrix representing $Z$ with respect to the same basis for $O_n(a,b,c)$ is tridiagonal. 
Since the $\BI$-module $O_n(a,b,c)$ is irreducible by assumption, these tridiagonal matrices are irreducible. The condition (i) follows.
\end{proof}

\begin{lem}
\label{lem:En(a,b,c)&LT}
Let $a,b,c\in \C$ and $n\geq 1$ be an odd integer.  
Suppose that the $\BI$-module $E_n(a,b,c)$ is irreducible. Then the following conditions are equivalent:
\begin{enumerate}
\item There exists a basis for $E_n(a,b,c)$ with respect to which the matrix representing $X$ is diagonal and the matrices representing $Y$ and $Z$ are irreducible tridiagonal.

\item $X$ is diagonalizable on $E_n(a,b,c)$.

\item $a\notin\{\frac{n-1}{2}-i\,|\,i=0,1,\ldots,n-1\}$.
\end{enumerate}
\end{lem}
\begin{proof}
We use the notations from Proposition \ref{prop:BImodule_evendim}.  
The argument is essentially the same as that in the proof of Lemma \ref{lem:On(a,b,c)&LT}, except for the following difference.
In the present case 
$$
\theta_{-1}=\frac{n}{2}-a+1,
\qquad 
\theta_{n+1}=\frac{n}{2}+a+1.
$$
The reasons why $\theta_{-1}$ is excluded from $\{\theta_i\}_{i=1}^n$
and why $\theta_{n+1}$ is excluded from $\{\theta_i\}_{i=0}^{\,n-1}$ differ slightly.
By construction, $\theta_{-1}\notin \{\theta_1,\theta_3,\ldots,\theta_n\}$ and $\theta_{n+1}\notin \{\theta_0,\theta_2,\ldots,\theta_{n-1}\}$. By (iii), $\theta_{-1}\notin \{\theta_2,\theta_4,\ldots,\theta_{n-1}\}$ and $\theta_{n+1}\notin \{\theta_1,\theta_3,\ldots,\theta_{n-2}\}$. 
\end{proof}

There exists a unique $\Z/3\Z$-action on $\BI$ such that each element of $\Z/3\Z$ acts as the algebra automorphism of $\BI$ described in the following table.
\begin{table}[H]
\centering
\extrarowheight=3pt
\begin{tabular}{c||rrr|rrr}
$u$  
&$X$ &$Y$ &$Z$ &$\kappa$ & $\lambda$ &$\mu$
\\
\midrule[1pt]
$(0\bmod 3)(u)$
&$X$  &$Y$   &$Z$ &$\kappa$  &$\lambda$  &$\mu$ 
\\
$(1\bmod 3)(u)$
& $Y$  &$Z$ &$X$ &$\lambda$ &$\mu$ &$\kappa$ 
\\
$(2\bmod 3)(u)$
& $Z$    &$X$ &$Y$ &$\mu$ &$\kappa$ &$\lambda$ 
\end{tabular}
\caption{The $\Z/3\Z$-action on $\BI$}
\label{t:BI&Z/3Z}
\end{table}

\begin{lem}
\label{lem:On(a,b,c)&Z/3Z}
Let $a,b,c\in \C$ and $n\geq 0$ be an even integer. Then the following conditions are equivalent:
\begin{enumerate}
\item The $\BI$-module $O_n(a,b,c)$ is irreducible.

\item The $\BI$-module $O_n(c,a,b)$ is irreducible.

\item The $\BI$-module $O_n(b,c,a)$ is irreducible.
\end{enumerate}
Suppose that {\rm (i)}--{\rm (iii)} hold. Then the $\BI$-modules 
$$
O_n(a,b,c),
\qquad 
O_n(c,a,b)^{1\bmod 3},
\qquad 
O_n(b,c,a)^{2\bmod 3}
$$ are isomorphic.
\end{lem}
\begin{proof}
Recall the irreducibility criterion for the $\BI$-module $O_n(a,b,c)$ from Theorem \ref{thm:BImodule_odddim_classification}. The equivalence of (i)--(iii) is immediate from the irreducibility criterion. 
By Lemmas \ref{lem:BImodule_odddim_trace} and \ref{lem:BImodule_odddim_criterion}, it is routine to verify the second assertion.
\end{proof}

\begin{lem}
\label{lem:En(a,b,c)&Z/3Z}
Let $a,b,c\in \C$ and $n\geq 1$ be an odd integer. Then the following conditions are equivalent:
\begin{enumerate}
\item The $\BI$-module $E_n(a,b,c)$ is irreducible.

\item The $\BI$-module $E_n(c,a,b)$ is irreducible.

\item The $\BI$-module $E_n(b,c,a)$ is irreducible.
\end{enumerate}
Suppose that {\rm (i)}--{\rm (iii)} hold. Then the $\BI$-modules 
$$E_n(a,b,c),\qquad 
E_n(c,a,b)^{1\bmod 3},
\qquad E_n(b,c,a)^{2\bmod 3}
$$ are isomorphic.
\end{lem}
\begin{proof}
Recall the irreducibility criterion for the $\BI$-module $E_n(a,b,c)$ from Theorem \ref{thm:BImodule_evendim_classification}. The equivalence of (i)--(iii) is immediate from the irreducibility criterion. 
By Lemmas \ref{lem:BImodule_evendim_trace} and \ref{lem:BImodule_evendim_criterion}, it is routine to verify the second assertion.
\end{proof}

\begin{thm}
\label{thm:On(a,b,c)&LT}  
Let $a,b,c\in \C$ and $n\geq 0$ be an even integer. Suppose that the $\BI$-module $O_n(a,b,c)$ is irreducible.  Then the following conditions are equivalent:
\begin{enumerate}
\item $X,Y,Z$ act on $O_n(a,b,c)$ as a Leonard triple.

\item $X,Y,Z$ are diagonalizable on $O_n(a,b,c)$.

\item $a,b,c\notin\{\frac{n-1}{2}-i\,|\,i=0,1,\ldots,n-1\}$.
\end{enumerate}
\end{thm}
\begin{proof} 
(i) $\Rightarrow$ (ii): Immediate from Definition \ref{defn:LT}.

(ii) $\Rightarrow$ (iii) $\Rightarrow$ (i): 
By Lemma \ref{lem:On(a,b,c)&LT} the following conditions are equivalent:
\begin{enumerate}
\item[$\bullet$] $X$ is diagonalizable on $O_n(a,b,c)$.

\item[$\bullet$] $a\notin\{\frac{n-1}{2}-i\,|\,i=0,1,\ldots,n-1\}$.

\item[$\bullet$] There exists a basis for $O_n(a,b,c)$ with respect to which the matrix representing $X$ is diagonal and those representing $Y$ and $Z$ are irreducible tridiagonal.
\end{enumerate}

By Table \ref{t:BI&Z/3Z} the actions of $X,Y,Z$ on the $\BI$-module $O_n(a,b,c)^{1\bmod 3}$ correspond respectively to the actions of $Y,Z,X$ on the $\BI$-module $O_n(a,b,c)$. 
By Lemma \ref{lem:On(a,b,c)&Z/3Z} the $\BI$-module $O_n(a,b,c)^{1\bmod 3}$ is isomorphic to the $\BI$-module $O_n(b,c,a)$. Hence, combined with Lemma \ref{lem:On(a,b,c)&LT}, the following conditions are equivalent:
\begin{enumerate}
\item[$\bullet$] $Y$ is diagonalizable on $O_n(a,b,c)$.

\item[$\bullet$] $X$ is diagonalizable on $O_n(a,b,c)^{1\bmod 3}$.

\item[$\bullet$] $X$ is diagonalizable on $O_n(b,c,a)$.

\item[$\bullet$] $b\notin\{\frac{n-1}{2}-i\,|\,i=0,1,\ldots,n-1\}$.

\item[$\bullet$] There exists a basis for $O_n(a,b,c)$ with respect to which the matrix representing $Y$ is diagonal and those representing $Z$ and $X$ are irreducible tridiagonal.
\end{enumerate}
By a similar argument the following conditions are equivalent:
\begin{enumerate}
\item[$\bullet$] $Z$ is diagonalizable on $O_n(a,b,c)$.

\item[$\bullet$] $X$ is diagonalizable on $O_n(a,b,c)^{2\bmod 3}$.

\item[$\bullet$] $X$ is diagonalizable on $O_n(c,a,b)$.

\item[$\bullet$] $c\notin\{\frac{n-1}{2}-i\,|\,i=0,1,\ldots,n-1\}$.

\item[$\bullet$] There exists a basis for $O_n(a,b,c)$ with respect to which the matrix representing $Z$ is diagonal and those representing $X$ and $Y$ are irreducible tridiagonal. 
\end{enumerate}
Therefore the implications (ii) $\Rightarrow$ (iii) and (iii) $\Rightarrow$ (i) follow.
\end{proof}

\begin{thm}
\label{thm:En(a,b,c)&LT}  
Let $a,b,c\in \C$ and $n\geq 1$ be an odd integer. Suppose that the $\BI$-module $E_n(a,b,c)$ is irreducible.  Then the following conditions are equivalent:
\begin{enumerate}
\item $X,Y,Z$ act on $E_n(a,b,c)$ as a Leonard triple.

\item $X,Y,Z$ are diagonalizable on $E_n(a,b,c)$.

\item $a,b,c\notin\{\frac{n-1}{2}-i\,|\,i=0,1,\ldots,n-1\}$.
\end{enumerate}
\end{thm}
\begin{proof} 
Relying on Lemmas \ref{lem:En(a,b,c)&LT} and \ref{lem:En(a,b,c)&Z/3Z}, the proof is analogous to that of Theorem \ref{thm:On(a,b,c)&LT}.
\end{proof}

\section{Proof for Theorem \ref{thm:V(1)&LT}}
\label{s:proofV(1)&LT}

Recall from Section \ref{s:UZ/2Zmodule} the $\U_{\Z/2\Z}$-modules $L_n^{\pm}$ for all $n \in \N$.
Let $V$ be a $\U_{\Z/2\Z}^{\otimes 2}$-module.
For any $\theta \in \C$, the subspace $V(\theta)$ of $V$ is defined in \eqref{V(theta)}. By Theorem \ref{thm:V(1)} the space $V(1)$ acquires a $\BI$-module structure.

\begin{lem}
\label{lem1:Lmn}
For any $m,n\in \N$ the vectors 
$$
v_i^{(m)}\otimes v_j^{(n)}
\qquad 
(0\leq i\leq m,\; 0\leq j\leq n,\; 2(i+j)=m+n-1)
$$ 
form a basis for $(L_m^+\otimes L_n^+)(1)$. 
\end{lem}
\begin{proof}
Immediate from \eqref{DeltaH} and \eqref{Ln:H}.
\end{proof}

\begin{lem}
\label{lem2:Lmn}
For any $m,n\in \N$ the following conditions are equivalent:
\begin{enumerate}
\item $(L_m^+\otimes L_n^+)(1)$ is nonzero.

\item $m+n$ is odd.
\end{enumerate}
\end{lem}
\begin{proof}
Immediate from Lemma \ref{lem1:Lmn}.
\end{proof}

\begin{lem}
\label{lem3:Lmn}
For any $m,n\in \N$ with $m\leq n$ and $m+n$ odd, the following statements hold:
\begin{enumerate}
\item The $\BI$-module $(L_m^+\otimes L_n^+)(1)$ is irreducible.

\item $X,Y,Z$ act on $(L_m^+\otimes L_n^+)(1)$ as a Leonard triple.
\end{enumerate}
\end{lem}
\begin{proof}
Since $m\leq n$ and $m+n$ is odd, it follows from Lemma \ref{lem1:Lmn} that the vectors 
$$
w_i=v_{m-i}^{(m)}\otimes v^{(n)}_{\frac{n-m-1}{2}+i}
\qquad 
(0\leq i\leq m)
$$
are a basis for $(L_m^+\otimes L_n^+)(1)$. For notational convenience, let $w_{-1}$ and $w_{m+1}$ denote the zero vector of $(L_m^+\otimes L_n^+)(1)$ in this proof. 
Recall from \eqref{X:V(1)} and \eqref{Y:V(1)} the actions of $X$ and $Y$ on  $(L_m^+\otimes L_n^+)(1)$.
Using \eqref{DeltaE}, \eqref{Deltarho}, \eqref{Ln:E},  \eqref{Ln:H} and \eqref{L+&L-:rho} a direct computation yields 
\begin{align}
X w_i &=\left(\frac{m+n+1}{2}-i \right) w_{m-i} +i\, w_{m-i+1}
\quad 
(0\leq i\leq m),
\label{X:Lmn}
\\
Y w_i &=\left(2i-m-\frac{1}{2}\right) w_i
\quad 
(0\leq i\leq m).
\label{Y:Lmn}
\end{align}

Let $W$ be a $\BI$-submodule of $(L_m^+\otimes L_n^+)(1)$.
We claim that if there is an integer $j$ with $0\leq j\leq m$ such that $w_j\in W$, then $w_{j-1}\in W$ and $w_{j+1}\in W$. Since $w_j\in W$ and $W$ is a $\BI$-module, the vector $Xw_j\in W$ and hence the right-hand side of (\ref{X:Lmn}) lies in $W$.
By (\ref{Y:Lmn}) the vectors $\{w_i\}_{i=0}^m$ are eigenvectors of $Y$ corresponding to distinct eigenvalues. It follows that 
$$
\left(\frac{m+n+1}{2}-j \right) w_{m-j}\in W;
\qquad 
j\, w_{m-j+1}\in W.
$$
If $j\neq 0$ then $j\, w_{m-j+1}\in W$ implies $w_{m-j+1}\in W$; if $j=0$ then $w_{m-j+1}=w_{m+1}=0\in W$. Hence 
$w_{m-j+1}\in W$ in either case.
Since $m\leq n$ and $0\leq j\leq m$, the scalar $\frac{m+n+1}{2}-j\neq 0$ and hence $w_{m-j}\in W$. So far we have seen that if $w_j\in W$ then $w_{m-j+1}\in W$ and $w_{m-j}\in W$. Applying this fact again, the claim follows.
To see (i), we assume that $W$ is nonzero and show that $W=(L_m^+\otimes L_n^+)(1)$. Since the eigenvalues of $Y$ on $(L_m^+\otimes L_n^+)(1)$ are pairwise distinct, there exists an integer $j$ with $0\leq j\leq m$ such that $w_j\in W$. Repeatedly applying the above claim, 
it follows that all vectors $\{w_i\}_{i=0}^m$ lie in $W$. The statement (i) follows.

By \eqref{X:Lmn} the trace of $X$ on $(L_m^+\otimes L_n^+)(1)$ is equal to 
\begin{align*}
\left\{
\begin{array}{ll}
\displaystyle\frac{n+1}{2} \qquad &\hbox{if $m$ is even},
\\
\displaystyle\frac{m+1}{2} \qquad &\hbox{if $m$ is odd}.
\end{array}
\right.
\end{align*}
By \eqref{Y:Lmn} the trace of $Y$ on $(L_m^+\otimes L_n^+)(1)$ is equal to $-\frac{m+1}{2}$. By \eqref{X:Lmn} and \eqref{Y:Lmn} the trace of $XY$ on $(L_m^+\otimes L_n^+)(1)$ is equal to 
\begin{align*}
\left\{
\begin{array}{ll}
-\displaystyle\frac{n+1}{4} \qquad &\hbox{if $m$ is even},
\\
\displaystyle\frac{m+1}{4} \qquad &\hbox{if $m$ is odd}.
\end{array}
\right.
\end{align*}
Recall from Section \ref{s:UZ/2Zmodule} that $\Lambda$ acts on $L_m^+$ and $L_n^+$ as scalar multiplication by $\frac{m(m+2)}{2}$ and $\frac{n(n+2)}{2}$, respectively. 
By Lemma \ref{lem:klm&Casimir_V(1)}(i) the elements $\kappa,\lambda,\mu$ act on $(L_m^+\otimes L_n^+)(1)$ as scalar multiplication by $0,0,\frac{m(m+2)}{2}-\frac{n(n+2)}{2}$, respectively. 
Using (\ref{BI:center1}) the trace of $Z$ on $(L_m^+\otimes L_n^+)(1)$ equals 
\begin{align*}
\left\{
\begin{array}{ll}
\displaystyle-\frac{n+1}{2} \qquad &\hbox{if $m$ is even},
\\
\displaystyle\frac{m+1}{2} \qquad &\hbox{if $m$ is odd}.
\end{array}
\right.
\end{align*}
By Lemmas \ref{lem:BImodule_odddim_criterion} and \ref{lem:BImodule_evendim_criterion} the irreducible $\BI$-module $(L_m^+\otimes L_n^+)(1)$ is isomorphic to 
\begin{align*}
\left\{
\begin{array}{ll}
\displaystyle O_m\left(\frac{n+1}{2},-\frac{m+1}{2},-\frac{n+1}{2}\right)
\qquad &\hbox{if $m$ is even},
\\
\displaystyle E_m\left(\frac{n+1}{2},\frac{m+1}{2},\frac{n+1}{2}\right)^{(-1,1)}
\qquad &\hbox{if $m$ is odd}.
\end{array}
\right.
\end{align*}
The statement (ii) now follows from Theorems \ref{thm:On(a,b,c)&LT} and \ref{thm:En(a,b,c)&LT}.
\end{proof}

Let $D_4$ denote the dihedral group of symmetries of a square. Recall that $D_4$ admits a presentation with generators $\sigma,\tau$ and the relations 
\begin{gather*}
\sigma^2=1,
\qquad 
\tau^4=1,
\qquad 
(\sigma\tau)^2=1.
\end{gather*}
Recall the $\{\pm 1\}^2$-action on $\BI$ from Section \ref{s:BImodule}. 
There exists a unique group homomorphism $\phi:D_4\to \{\pm 1\}^2$ that sends
\begin{eqnarray}
\label{phi}
\sigma &\mapsto & (1,-1),
\qquad 
\tau \;\;\mapsto \;\; (-1,-1).
\end{eqnarray}
By Theorem \ref{thm:UZ/2Z_presentation} there exists a unique $D_4$-action on $\U_{\Z/2\Z}^{\otimes 2}$ such that $\sigma$ and $\tau$ act as the algebra automorphisms of $\U_{\Z/2\Z}^{\otimes 2}$ described in the following table:
 \begin{table}[H]
\centering
\extrarowheight=3pt
\begin{tabular}{c||cccc|cccc}
$u$  
&$E\otimes 1$ &$F\otimes 1$ &$H\otimes 1$ &$\rho\otimes 1$
&$1\otimes E$ &$1\otimes F$ &$1\otimes H$ &$1\otimes \rho$
\\

\midrule[1pt]

$\sigma(u)$
&$1\otimes E$ &$1\otimes F$ &$1\otimes H$ &$1\otimes \rho$
&$E\otimes 1$ &$F\otimes 1$ &$H\otimes 1$ &$\rho\otimes 1$
\\

$\tau(u)$
&$1\otimes E$ &$1\otimes F$ &$1\otimes H$ &$1\otimes \rho$
&$E\otimes 1$ &$F\otimes 1$ &$H\otimes 1$ &$-\rho\otimes 1$
\end{tabular}
\caption{The $D_4$-action on $\U_{\Z/2\Z}^{\otimes 2}$}
\label{t:UZ/2Z&D4}
\end{table}

\begin{lem}
\label{lem:K4&D4}
For any $\U_{\Z/2\Z}^{\otimes 2}$-module $V$ and any $g\in D_4$, the $\BI$-module $V^g(1)$ coincides with the $\BI$-module $V(1)^{\phi(g)}$.
\end{lem}
\begin{proof}
Applying Table \ref{t:UZ/2Z&D4} to (\ref{X:V(1)})--(\ref{Z:V(1)}), we observe that the actions of $X,Y,Z$ on $(V^\sigma)(1)$ correspond respectively to the actions of $X,-Y,-Z$ on $V(1)$. Likewise, the actions of $X,Y,Z$ on $(V^\tau)(1)$ correspond respectively to the actions of $-X,-Y,Z$ on $V(1)$. 
Thus, by Table \ref{t:BI&pm1^2} and (\ref{phi}), the lemma is true when $g=\sigma$ and $g=\tau$.
Since the group $D_4$ is generated by $\sigma$ and $\tau$, the lemma follows.
\end{proof}

We now proceed to the proof for Theorem \ref{thm:V(1)&LT}:

\begin{proof}[Proof of Theorem \ref{thm:V(1)&LT}]
The implications (ii) $\Rightarrow$ (i) and (iii) $\Rightarrow$ (i) are trivial. It remains to show (i) $\Rightarrow$ (ii) and (i) $\Rightarrow$ (iii).

Let $V$ be a finite-dimensional irreducible $\U_{\Z/2\Z}^{\otimes 2}$-module. 
By Theorem \ref{thm:UZ/2Zmodule_classification} 
the $\U_{\Z/2\Z}$-modules $L_n^{\pm}$ for all $n\in \N$ are all finite-dimensional irreducible $\U_{\Z/2\Z}$-modules up to isomorphism. 
According to \cite[Theorem 3.10.2(ii)]{etingof2011}, the $\U_{\Z/2\Z}^{\otimes 2}$-module $V$ is of one of the following forms up to isomorphism:
\begin{gather*}
L_m^+\otimes L_n^+, 
\quad L_m^+\otimes L_n^-,
\quad L_m^-\otimes L_n^+,
\quad L_m^-\otimes L_n^-,
\\
L_n^+\otimes L_m^+, 
\quad L_n^+\otimes L_m^-, 
\quad L_n^-\otimes L_m^+, 
\quad L_n^-\otimes L_m^-
\end{gather*}
for some $m,n\in \N$ with $m\leq n$. 
By (\ref{L+&L-:rho}) and Table \ref{t:UZ/2Z&D4} these eight $\U_{\Z/2\Z}^{\otimes 2}$-modules are isomorphic to  
\begin{gather*}
(L_m^+\otimes L_n^+)^1, 
\quad (L_m^+\otimes L_n^+)^{\sigma\tau},
\quad (L_m^+\otimes L_n^+)^{\sigma\tau^{-1}},
\quad (L_m^+\otimes L_n^+)^{\tau^2},
\\
(L_m^+\otimes L_n^+)^\sigma, 
\quad (L_m^+\otimes L_n^+)^\tau, 
\quad (L_m^+\otimes L_n^+)^{\tau^{-1}}, 
\quad (L_m^+\otimes L_n^+)^{\sigma\tau^2},
\end{gather*}
respectively.

Consequently, there exist $m,n\in \N$ with $m\leq n$ and an element $g\in D_4$ such that the $\U_{\Z/2\Z}^{\otimes 2}$-module $V$ is isomorphic to $(L_m^+\otimes L_n^+)^{g}$. By Lemma \ref{lem:K4&D4} the $\BI$-module $V(1)$ is isomorphic to $(L_m^+\otimes L_n^+)(1)^{\phi(g)}$. The implications (i) $\Rightarrow$ (ii) and (i) $\Rightarrow$ (iii) then follow from Lemmas \ref{lem2:Lmn} and \ref{lem3:Lmn} as well as Table \ref{t:BI&pm1^2}.
\end{proof}

\section{Proof for Theorem \ref{thm:BI->T}} 
\label{s:proofBI->T}

Recall some notations from Section \ref{s:intro}.
For a set $X$, let $\C^X$ be the vector space with basis $X$. Let $\Omega$ be a finite set and $2^\Omega$ be the power set of $\Omega$.
In this section we adopt the symbol $\subsetdot$ to denote the covering relation in $(2^\Omega, \subseteq)$.

\begin{thm}
[Theorem 13.2, \cite{hypercube2002}]
\label{thm:Go}
There exists a unique $\U$-module $\C^{2^\Omega}$ given by 
\begin{align}
    Ex&=\sum_{y \subsetdot x} y,
    \label{E:2^Omega}
    \\
    Fx&=\sum_{x \subsetdot y} y,
    \label{F:2^Omega}    
    \\
    Hx&=(|\Omega|-2|x|) x
    \label{H:2^Omega}    
\end{align}
for all $x\in 2^\Omega$, where the two sums are taken over all $y \in 2^\Omega$ with $y \subsetdot x$ and $x \subsetdot y$, respectively.
\end{thm}

Let $x_0\in 2^{\Omega}$ be given. By Theorem \ref{thm:Go} the spaces $\C^{2^{x_0}}$ and $\C^{2^{\Omega\setminus x_0}}$ are $\U$-modules. Hence $\C^{2^{\Omega\setminus x_0}}\otimes \C^{2^{x_0}}$ is a $\U^{\otimes 2}$-module. 
Define a linear map $\iota(x_0):\C^{2^\Omega}\to \C^{2^{\Omega\setminus x_0}}\otimes \C^{2^{x_0}}$ by 
\begin{eqnarray}
\label{iota(x0)}
x&\mapsto& (x\setminus x_0) \otimes (x\cap x_0) 
    \qquad 
 \hbox{for all $x\in 2^\Omega$}.
\end{eqnarray}
Note that $\iota(x_0)$ is a linear isomorphism.

\begin{lem}
[Lemma 5.4, \cite{Huang:CG&Johnson}]
\label{lem:diagram}
For any $x_0\in 2^\Omega$ and $u\in \U$ the following diagram is commutative:
\begin{figure}[H]
\centering
\begin{tikzpicture}
\matrix(m)[matrix of math nodes,
row sep=4em, column sep=4em,
text height=1.5ex, text depth=0.25ex]
{
\C^{2^\Omega}
&\C^{2^{\Omega\setminus x_0}}\otimes \C^{2^{x_0}}\\
\C^{2^\Omega}
&\C^{2^{\Omega\setminus x_0}}\otimes \C^{2^{x_0}}\\
};
\path[->,font=\scriptsize,>=angle 90]
(m-1-1) edge node[left] {$u$} (m-2-1)
(m-1-1) edge node[above] {$\iota(x_0)$} (m-1-2)
(m-2-1) edge node[below] {$\iota(x_0)$} (m-2-2)
(m-1-2) edge node[right] {$\Delta(u)$} (m-2-2);
\end{tikzpicture}
\end{figure}
\end{lem}

By pulling back via the comultiplication of $\U$, every $\U\otimes \U$-module can be regarded as a $\U$-module. 
Thus, Lemma \ref{lem:diagram} can also be interpreted as asserting that $\iota(x_0)$ realizes a $\U$-module isomorphism from $\C^{2^\Omega}$ to $\C^{2^{\Omega\setminus x_0}}\otimes \C^{2^{x_0}}$.

\begin{thm}
\label{thm:UZ/2Z_Go&diagram}
\begin{enumerate}
\item There exists a unique $\U_{\Z/2\Z}$-module $\C^{2^\Omega}$
satisfying {\rm\eqref{E:2^Omega}--\eqref{H:2^Omega}} and 
\begin{gather}
\label{rho:2^Omega}
    \rho\, x=\Omega \setminus x     
    \qquad
    \hbox{for all $x\in 2^\Omega$}.
\end{gather}

\item For any $x_0\in 2^\Omega$ and $u\in \U_{\Z/2\Z}$ the following diagram is commutative:
\begin{figure}[H]
\centering
\begin{tikzpicture}
\matrix(m)[matrix of math nodes,
row sep=4em, column sep=4em,
text height=1.5ex, text depth=0.25ex]
{
\C^{2^\Omega}
&\C^{2^{\Omega\setminus x_0}}\otimes \C^{2^{x_0}}\\
\C^{2^\Omega}
&\C^{2^{\Omega\setminus x_0}}\otimes \C^{2^{x_0}}\\
};
\path[->,font=\scriptsize,>=angle 90]
(m-1-1) edge node[left] {$u$} (m-2-1)
(m-1-1) edge node[above] {$\iota(x_0)$} (m-1-2)
(m-2-1) edge node[below] {$\iota(x_0)$} (m-2-2)
(m-1-2) edge node[right] {$\Delta(u)$} (m-2-2);
\end{tikzpicture}
\end{figure}
\end{enumerate} 
\end{thm}
\begin{proof}
(i): By Theorem \ref{thm:Go} the relations \eqref{[HE]}--\eqref{[EF]} hold on $\C^{2^\Omega}$. According to Theorem \ref{thm:UZ/2Z_presentation}, it remains to verify the relations \eqref{UZ/2Z:Hrho}--\eqref{UZ/2Z:rho^2} on $\C^{2^\Omega}$.

Let $x \in 2^\Omega$ be given.  
By \eqref{rho:2^Omega} the left-hand side of (\ref{UZ/2Z:rho^2}) fixes $x$, so the relation \eqref{UZ/2Z:rho^2} holds on $\C^{2^\Omega}$. 
Evaluating both sides of \eqref{UZ/2Z:Hrho} at $x$ using \eqref{H:2^Omega} and \eqref{rho:2^Omega}, both sides are equal to $(|\Omega|-2|x|)\cdot (\Omega\setminus x)$. Hence the relation \eqref{UZ/2Z:Hrho} holds on $\C^{2^\Omega}$. 
Evaluating the left-hand side of \eqref{UZ/2Z:rhoE} at $x$ using \eqref{E:2^Omega} and \eqref{rho:2^Omega} gives
$
\sum_{y \subsetdot x} \Omega \setminus y$. 
Evaluating the right-hand side at $x$ using \eqref{F:2^Omega} and \eqref{rho:2^Omega} gives
$
\sum_{\Omega \setminus x \subsetdot y} y$.
Both expressions coincide with
$$
\sum_{\Omega \setminus y \subsetdot x} y.
$$
Therefore the relation \eqref{UZ/2Z:rhoE} holds on $\C^{2^\Omega}$. The statement (i) follows.

(ii): Let $x \in 2^\Omega$ be given.  
By \eqref{iota(x0)} and \eqref{rho:2^Omega} the composition $\iota(x_0)\circ \rho$ maps $x$ to 
$$
\big((\Omega\setminus x)\setminus x_0\big)
\,\otimes\,
\big((\Omega\setminus x)\cap x_0\big).
$$
By \eqref{Deltarho}, \eqref{iota(x0)} and \eqref{rho:2^Omega} the composition $\Delta(\rho)\circ \iota(x_0)$ maps $x$ to
$$
\big((\Omega\setminus x_0)\setminus (x\setminus x_0)\big)
\,\otimes\,
\big( x_0\setminus (x\cap x_0)\big).
$$ 
The first factors of both tensor products are equal to $\Omega\setminus (x\cup x_0)$, and the second factors are equal to $x_0\setminus x$. 
Therefore the two tensor products coincide.
Combined with Lemma \ref{lem:diagram}, the statement (ii) follows.
\end{proof}

We now regard $\C^{2^\Omega}$ as a $\U_{\Z/2\Z}^{\otimes 2}$-module by identifying it with $\C^{2^{\Omega\setminus x_0}} \otimes \C^{2^{x_0}}$ via the linear isomorphism $\iota(x_0)$.
The resulting $\U_{\Z/2\Z}^{\otimes 2}$-module structure depends on the choice of $x_0$, and we denote this module by
$$
\C^{2^\Omega}(x_0).
$$

\begin{lem}
\label{lem:2^Omega(x0)}
For any $x_0\in 2^\Omega$, the following equations hold for all $x\in 2^\Omega$ in the $\U_{\Z/2\Z}^{\otimes 2}$-module $\C^{2^\Omega}(x_0)$:
\begin{enumerate}
\item $(H\otimes 1) x=(|\Omega\setminus x_0|-2|x\setminus x_0|) x$.

\item $(1\otimes H) x=(|x_0|-2|x\cap x_0|) x$.

\item $\Delta(H) x=(|\Omega|-2|x|) x$.

\item $\Delta(E\rho) x=\sum_{y\subsetdot \Omega\setminus x} y$, 
where the sum is over all $y\in 2^\Omega$ with $y\subsetdot \Omega\setminus x$.
\end{enumerate}
\end{lem}
\begin{proof}
(i), (ii): These are immediate from \eqref{H:2^Omega} and the construction of $\C^{2^\Omega}(x_0)$.

(iii): This is immediate from (i), (ii) and \eqref{DeltaH}. Alternatively, the equation (iii) can be deduced directly from \eqref{H:2^Omega} by applying Theorem \ref{thm:UZ/2Z_Go&diagram}(ii) with $u = H$.

(iv): Using \eqref{E:2^Omega} and \eqref{rho:2^Omega} shows that $(E\rho) x$ equals the right-hand side of (iv) for all $x\in 2^\Omega$. 
Applying Theorem \ref{thm:UZ/2Z_Go&diagram}(ii) with $u=E\rho$, the statement (iv) follows.
\end{proof}

We next turn to the proof for Theorem \ref{thm:BI->T}:

\begin{proof}[Proof of Theorem \ref{thm:BI->T}] 
Assume that $d\ge 1$ is an integer and $|\Omega|=2d+1$, which are the hypotheses of Theorem \ref{thm:BI->T}. 
By Theorem \ref{thm:V(1)}, the space $\C^{2^\Omega}(x_0)(1)$ carries a $\BI$-module structure on which $X$ and $Y$ act as $\Delta(E\rho)$ and $\frac{H\otimes 1-1\otimes H}{2}$, respectively.
It follows from Lemma \ref{lem:2^Omega(x0)}(iii) that 
$$
\C^{2^\Omega}(x_0)(1)=\C^{{\Omega\choose d}}
\qquad 
\hbox{for any $x_0\in 2^\Omega$}.
$$
Now let $x_0\in {\Omega\choose d}$. 

(i): Let $x\in {\Omega\choose d}$.  For any $y\in 2^\Omega$, the condition $y\subsetdot \Omega\setminus x$ is equivalent to $x\cap y=\emptyset$ and $y\in {\Omega\choose d}$. Hence, comparing \eqref{A} with Lemma \ref{lem:2^Omega(x0)}(iv), the operator $X$ on $\C^{2^\Omega}(x_0)(1)$ coincides with $\A$. 

Since $|\Omega|=2d+1$ and $|x_0|=d$, it follows from Lemma \ref{lem:2^Omega(x0)}(i), (ii) that $Yx$ is a scalar multiple of $x$ with scalar 
\begin{gather}
\label{e1:Yx}
\frac{1}{2}-|x\setminus x_0|+|x\cap x_0|.
\end{gather}
Since $x\cap x_0=x_0\setminus (x_0\setminus x)$, we have $|x\cap x_0|=d-|x_0\setminus x|$. Hence \eqref{e1:Yx} can be rewritten as
$$
d+\frac{1}{2}-(|x\setminus x_0|+|x_0\setminus x|).
$$
By comparing this with \eqref{A^*}, the operator $Y$ on $\C^{2^\Omega}(x_0)(1)$ coincides with 
$$
\frac{d+1}{2d+1} \A^*(x_0)+\frac{1}{2(2d+1)}.
$$
By Lemma \ref{lem:klm&Casimir_V(1)}(i) the element $\kappa$ vanishes on $\C^{2^\Omega}(x_0)(1)$. Using (\ref{BI:center1}) the operator $Z$ on $\C^{2^\Omega}(x_0)(1)$ coincides with 
$$
\frac{d+1}{2d+1} \{\A,\A^*(x_0)\}+\frac{1}{2d+1}\A.
$$

By the preceding arguments, the existence of the homomorphism $\BI\to \T(x_0)$ follows. The uniqueness is immediate since $\BI$ is generated by $X,Y,Z$. The homomorphism is surjective
since $\T(x_0)$ is generated by $\A,\A^*(x_0)$. The statement (i) follows.

(ii): By Theorem \ref{thm:UZ/2Zmodule_semisimple} the $\U_{\Z/2\Z}$-modules $\C^{2^{\Omega\setminus x_0}}$ and $ \C^{2^{x_0}}$ are completely reducible. Hence the $\U_{\Z/2\Z}^{\otimes 2}$-module $\C^{2^{\Omega\setminus x_0}}\otimes \C^{2^{x_0}}$, and consequently $\C^{2^\Omega}(x_0)$, is completely reducible. 
 The statement (ii) now follows from Theorem \ref{thm:V(1)&LT}.
\end{proof}

\section{The Clebsch--Gordan rule for $\U_{\Z/2\Z}$}
\label{s:CG}

Recall the $\U$-module $L_n$ ($n\in \N$) from (\ref{Ln:E})--(\ref{Ln:H}). The $\U$-module $L_n$ satisfies the following universal property:

\begin{lem}
\label{lem:Ln_universal}
Let $V$ denote a finite-dimensional $\U$-module. 
Suppose that there exist an integer $n\in \N$ and a vector $v\in V$ such that $Hv=nv$ and $Ev=0$. Then there exists a unique $\U$-module homomorphism $L_n\to V$ that maps $v_0^{(n)}$ to $v$.
\end{lem}
\begin{proof}
See\cite[Theorem V.4.4]{kassel} for example.
\end{proof}

For any integers $i$ and $n$ with $0\leq i\leq n$, the symbol
$$
{n\choose i}=\frac{n(n-1)\cdots (n-i+1)}{i!}
$$
denotes the binomial coefficient.

\begin{lem}
\label{lem:Lm+n-2p}
Let $m,n\in \N$. For any integer $p$ with $0\leq p\leq \min\{m,n\}$, the following statements hold:
\begin{enumerate}
\item There exists a unique $\U$-module homomorphism $f:L_{m+n-2p}\to L_m\otimes L_n$ such that 
\begin{eqnarray*}
v_0^{(m+n-2p)} 
&\mapsto &
\sum_{i=0}^p (-1)^i {p\choose i} 
\, v_i^{(m)}\otimes v_{p-i}^{(n)}.
\end{eqnarray*}
In particular $f$ is injective.

\item $f\big(v_{m-p}^{(m+n-2p)}\big)$ is a linear combination of $\{v_{m-i}^{(m)}\otimes v_i^{(n)}\}_{i=1}^{\min\{m,n\}}$ plus
$$
(-1)^p {m+n-2p \choose m-p}^{-1} v_m^{(m)}\otimes v_0^{(n)}.
$$

\item $f\big(v_{n-p}^{(m+n-2p)}\big)$ is a linear combination of $\{v_{i}^{(m)}\otimes v_{n-i}^{(n)}\}_{i=1}^{\min\{m,n\}}$ plus
$$
{m+n-2p \choose n-p}^{-1} v_0^{(m)}\otimes v_n^{(n)}.
$$
\end{enumerate}
\end{lem}
\begin{proof}
Using Lemma~\ref{lem:Ln_universal} the existence and uniqueness of $f$ is straightforward to verify. Since the $\U$-module $L_{m+n-2p}$ is irreducible, the homomorphism $f$ is injective.
The statements (ii) and (iii) follow from direct computation using  (i), \eqref{DeltaF} and \eqref{Ln:F}.
\end{proof}

Let $m,n\in \N$ be given.
By Lemma \ref{lem:Lm+n-2p}(i) the $\U$-module $L_m\otimes L_n$ contains the $\U$-submodules isomorphic to $L_{m+n-2p}$ for all integers $p$ with $0\leq p\leq \min\{m,n\}$. A dimension argument then shows that the $\U$-module $L_m\otimes L_n$ is isomorphic to
\begin{gather}
\label{CG_U}
\bigoplus_{p=0}^{\min\{m,n\}}
L_{m+n-2p}.
\end{gather}
This decomposition is known as the {\it Clebsch--Gordan rule} for $\U$. 
As a consequence, for each integer $p$ with $0\leq p\leq \min\{m,n\}$, there is a unique $\U$-submodule of $L_m\otimes L_n$ that is isomorphic to $L_{m+n-2p}$.

In order to formulate the Clebsch--Gordan rule for $\U_{\Z/2\Z}$, the signs $+$ and $-$ in the $\U_{\Z/2\Z}$-modules $L_n^+$ and $L_n^-$ ($n\in \N$) are identified with the real scalars $+1$ and $-1$, respectively.

\begin{thm}
\label{thm:CG_UZ/2Z} 
Suppose that $m,n\in \N$ and $\delta,\e \in \{\pm 1\}$. Then the $\U_{\Z/2\Z}$-module $L_m^\delta \otimes L_n^{\e}$ is isomorphic to 
\begin{gather*}
  \bigoplus_{p=0}^{\min\{m,n\}}  
  L_{m+n-2p}^{(-1)^p\cdot \delta \cdot \e}.
\end{gather*}
\end{thm}
\begin{proof}
As a $\U$-module, $L_m^{\delta}\otimes L_{n}^{\e}$ is isomorphic to $L_m\otimes L_n$. 
By the Clebsch--Gordan rule for $\U$, the $\U$-module $L_m^{\delta}\otimes L_{n}^{\e}$ is isomorphic to (\ref{CG_U}). 
Fix an integer $p$ with $0\leq p\leq \min\{m,n\}$. 
Let 
$$
f: L_{m+n-2p}\to L_m^{\delta}\otimes L_{n}^{\e}
$$ 
be the injective $\U$-module homomorphism given by Lemma \ref{lem:Lm+n-2p}(i).

Let $W$ denote the image of $f$. Then the $\U$-module $W$ is isomorphic to $L_{m+n-2p}$.
By Lemma \ref{lem1:W&rhoW}(i), the space $\rho W$ is a $\U$-submodule of $L_m^{\delta}\otimes L_{n}^{\e}$. Moreover, Lemma \ref{lem2:W&rhoW}(i) implies that $\rho W$ is isomorphic to $L_{m+n-2p}$. 
Thus 
$$
W=\rho W.
$$
By Lemma \ref{lem1:W&rhoW}(ii) the space $W$ is a $\U_{\Z/2\Z}$-submodule of $L_m^{\delta}\otimes L_{n}^{\e}$.
It then follows from Lemmas \ref{lem:Ln+&Ln-}(ii) and  \ref{lem2:W&rhoW}(ii) that there exists a unique $\epsilon\in \{\pm 1\}$ for which $W$ is isomorphic to $L_{m+n-2p}^{\epsilon}$. It suffices to prove that 
\begin{gather}
\label{epsilon}
\epsilon=(-1)^p\delta\e.
\end{gather}

Since the $\U_{\Z/2\Z}$-module $W$ is isomorphic to $L_{m+n-2p}^{\epsilon}$, there exists a $\U_{\Z/2\Z}$-module isomorphism from $L_{m+n-2p}^{\epsilon}$ to $W$.
As this map is also a $\U$-module isomorphism, by Schur's lemma it must be a nonzero scalar multiple of $f$. 
Hence the map
$$
f:L_{m+n-2p}^{\epsilon}\to W
$$
is a $\U_{\Z/2\Z}$-module isomorphism. In particular $f$ commutes with $\rho$. Consider the relation
\begin{gather}
\label{rhof=frho}
(\rho f)\big(v_{m-p}^{(m+n-2p)}\big)=(f\rho)\big(v_{m-p}^{(m+n-2p)}\big).
\end{gather}
By \eqref{L+&L-:rho} and Lemma \ref{lem:Lm+n-2p}(iii), the right-hand side of \eqref{rhof=frho} is a linear combination of $\{v_i^{(m)}\otimes v_{n-i}^{(n)}\}_{i=1}^{\min\{m,n\}}$ plus 
\begin{gather}
\label{RHS:rhof=frho}
\epsilon\cdot  {m+n-2p \choose n-p}^{-1} v_0^{(m)}\otimes v_n^{(n)}.
\end{gather}
Similarly, by \eqref{L+&L-:rho} and Lemma \ref{lem:Lm+n-2p}(ii), the left-hand side of \eqref{rhof=frho} is a linear combination of $\{v_i^{(m)}\otimes v_{n-i}^{(n)}\}_{i=1}^{\min\{m,n\}}$ plus 
\begin{gather}
\label{LHS:rhof=frho}
(-1)^p\delta \e\cdot {m+n-2p \choose m-p}^{-1} v_0^{(m)}\otimes v_n^{(n)}.
\end{gather}
Since ${m+n-2p \choose n-p}={m+n-2p \choose m-p}$, comparing \eqref{RHS:rhof=frho} and \eqref{LHS:rhof=frho} yields  \eqref{epsilon}. The result follows. 
\end{proof}

Finally, we apply the Clebsch--Gordan rule for $\U_{\Z/2\Z}$ (Theorem~\ref{thm:CG_UZ/2Z}) to decompose the $\U_{\Z/2\Z}$-module $\C^{2^\Omega}$ as described in Theorem~\ref{thm:UZ/2Z_Go&diagram}(i).
For any vector space $V$ and any positive integer $n$, we write $n\cdot V$ for the direct sum of $n$ copies of $V$.

\begin{cor}
\label{cor:2^Omega_dec} 
The $\U_{\Z/2\Z}$-module $\C^{2^\Omega}$ is isomorphic to 
\begin{gather}
\label{2^Omega_dec}
\bigoplus_{i=0}^{\floor{\frac {|\Omega|}2}} 
\frac{|\Omega|-2i+1}{|\Omega|-i+1}{|\Omega|\choose i}
\cdot  
L_{|\Omega|-2i}^{(-1)^i}.
\end{gather}
\end{cor}
\begin{proof} 
We prove the decomposition by induction on $|\Omega|$. By Theorem~\ref{thm:UZ/2Z_Go&diagram}(i), the decomposition holds for $|\Omega|=0$ and $|\Omega|=1$, since in these cases $\C^{2^\Omega}$ is isomorphic to $L_0^+$ and $L_1^+$ respectively.

Assume that $|\Omega|\geq 2$. For notational convenience, set $D=|\Omega|$ and define
$$
m_i(k)=\frac{k-2i+1}{k-i+1}{k\choose i}
$$
for any integers $i$ and $k$ with $0\leq i\leq k$.
Let $x_0\in 2^\Omega$ with $|x_0|=1$. By Theorem \ref{thm:UZ/2Z_Go&diagram}(ii) the $\U$-module $\C^{2^\Omega}$ is isomorphic to $\C^{2^{\Omega\setminus x_0}}\otimes \C^{2^{x_0}}$. 
As noted earlier, the $\U$-module $\C^{2^{x_0}}$ is isomorphic to $L_1^+$. 
By the induction hypothesis, the $\U$-module $\C^{2^{\Omega\setminus x_0}}$ is isomorphic to 
$$
\bigoplus_{i=0}^{\floor{\frac {D-1}{2}}} 
m_i(D-1)
\cdot  
L_{D-2i-1}^{(-1)^i}.
$$
The distributive law of $\otimes$ over $\oplus$ yields the following decomposition of $\C^{2^{\Omega\setminus x_0}}\otimes \C^{2^{x_0}}$: 
\begin{gather}
\label{2^Omega_dec'}
\bigoplus_{i=0}^{\floor{\frac {D-1}2}} 
m_i(D-1)
\cdot  
(L_{D-2i-1}^{(-1)^i}\otimes L_1^+).
\end{gather}
Applying Theorem \ref{thm:CG_UZ/2Z} the $\U$-module $L_{D-2i-1}^{(-1)^i}\otimes L_1^+$ is isomorphic to 
\begin{gather*}
\left\{
\begin{array}{ll}
L_{D-2i}^{(-1)^i} \oplus L_{D-2i-2}^{(-1)^{i+1}}
\qquad &\hbox{if $i\leq \frac{D}{2}-1$},
\\
L_1^{(-1)^i}
\qquad &\hbox{otherwise}
\end{array}
\right.
\qquad 
\left(\textstyle{0\leq i\leq \floor{\frac {D-1}{2}}}\right).
\end{gather*}
If $D$ is even then the $\U_{\Z/2\Z}$-module (\ref{2^Omega_dec'}) is isomorphic to 
\begin{gather}
\label{evenD}
m_0(D-1)\cdot L_D^+
\oplus 
m_{\frac{D}{2}-1}(D-1)\cdot L_0^{(-1)^{\frac{D}{2}}}
\oplus 
 \bigoplus_{i=1}^{\frac{D}{2}-1}
(m_{i-1}(D-1)+m_i(D-1))\cdot L_{D-2i}^{(-1)^i}.
\end{gather}
If $D$ is odd then the $\U_{\Z/2\Z}$-module (\ref{2^Omega_dec'}) is isomorphic to 
\begin{gather}
\label{oddD}
m_0(D-1)\cdot L_D^+\oplus \bigoplus_{i=1}^{\frac{D-1}{2}}
(m_{i-1}(D-1)+m_i(D-1))\cdot L_{D-2i}^{(-1)^i}.
\end{gather}
Observe that $m_0(D)=m_0(D-1)=1$ and 
$$
m_{i-1}(D-1)+m_i(D-1)=m_i(D)
\qquad 
(1\leq i\leq D-1).
$$
Note that if $D$ is even and $i=\frac{D}{2}$, then $m_i(D-1)=0$. Substituting these equalities into \eqref{evenD} and \eqref{oddD} yields the decomposition \eqref{2^Omega_dec}. The result follows.
\end{proof}

\subsection*{Funding statement}
The research was supported by the National Science and Technology Council of Taiwan under the project NSTC 114-2115-M-008-011.

\subsection*{Conflict of interest statement}
The authors declare that they have no conflicts of interest.

\subsection*{Data availability statement}
No datasets were generated or analyzed during the current study.

\bibliographystyle{amsplain}
\bibliography{MP}

\end{document}